\documentclass{article}
\usepackage{fullpage}

\usepackage{amsmath}
\usepackage{amsfonts}
\usepackage{amsthm}
\usepackage{amssymb}
\usepackage{multirow}
\usepackage{algorithm,algorithmic}
\usepackage{mathrsfs}
\usepackage{color,xcolor}
\usepackage{ifpdf}
\usepackage{psfrag}
\usepackage{graphicx,graphics,subfigure,epsfig}
\usepackage{url}
\usepackage{hyperref}
\usepackage{ifpdf}
\usepackage{psfrag}
\usepackage{graphicx,graphics,subfigure,epsfig}
\usepackage{pgf,tikz}
\usetikzlibrary{ bending, patterns,calc}
\usepackage{pgfplots}
\usepackage{xspace}
\usepackage{xargs}
\usepackage{float}
 \usepackage{ booktabs,array}
\usepackage{extarrows}

\usepackage{appendix}
\usepackage{chngcntr}
\AtBeginEnvironment{subappendices}{%
\chapter*{Appendix}
\addcontentsline{toc}{chapter}{Appendices}
\counterwithin{figure}{section}
\counterwithin{table}{section}
}
\usepackage{bm}
\usepackage{dutchcal}
\usepackage{stmaryrd}

\usepackage[colorinlistoftodos,prependcaption,textsize=tiny]{todonotes}

\newcommand{\yolo}[1]{\textcolor{black}{#1}}
\newcommand{\yoloo}[1]{\textcolor{black}{#1}}

\newtheorem{theorem}{Theorem}[section]

\newtheorem{lemma}[theorem]{Lemma}

\newtheorem{proposition}[theorem]{Proposition}
\newtheorem{remark}[theorem]{Remark}

\newcommand{\dpar}[2]{\dfrac{\partial #1}{\partial #2}}

\newcommand{\R}{\mathbb R}

\renewcommand{\P}{\mathbb P}

\newcommand{\DD}{\mathcal D}

\newcommand{\FF}{\mathcal F}

\newcommand{\II}{\mathcal I}

\newcommand{\TT}{\mathcal{T}}

\newcommand{\bba}{\mathbf{a}}

\newcommand{\bbf}{\mathbf{f}}
\newcommand{\bbg}{\mathbf{g}}

\newcommand{\bbm}{\mathbf{m}}
\newcommand{\bbn}{\mathbf{n}}

\newcommand{\bbu}{\mathbf{u}}
\newcommand{\bbv}{\mathbf{v}}

\newcommand{\bbx}{\mathbf{x}}

\newcommand{\bbK}{\mathbf{K}}

\newcommand{\bbN}{\mathbf{N}}

\newcommand{\bbPhi}{{\mathbf {\Phi}}}

\newcommand{\Id}{\mathbb{I}}

\newcommand{\hbbg}{\hat{\mathbf{g}}}
\newcommand{\hbbf}{\hat{\mathbf{f}}}

\usepackage{amsopn}

\newcommand*\xbar[1]{%
  \hbox{%
    \vbox{%
      \hrule height 0.5pt % The actual bar
      \kern0.4ex%         % Distance between bar and symbol
      \hbox{%
        \kern-0.05em%      % Shortening on the left side
        \ensuremath{#1}%
        \kern-0.00em%      % Shortening on the right side
      }%
    }%
  }%
}
\newcommand{\bsigma}{{\bm {\sigma}}}

\newcommand{\bomega}{{\bm {\omega}}}

\newcommand{\bPhi}{{\bm \Phi}}

\usepackage[style=numeric,backend=biber,maxnames=999,isbn=false,url=false, doi=false]{biblatex}
	\title{Construction of entropy satisfying Active Flux-type methods}

\author{R. Abgrall  and Y. Liu\\
  Institute of Mathematics, University of Z\"urich  }
\date{\today}
\begin{document}
\maketitle
\begin{abstract}
This paper is devoted to the analysis of the entropy stability properties of Active Flux\yolo{-type} scheme for \yolo{a} hyperbolic system equipped with one entropy inequality. We also refine, with respect to \cite{Abgrall_AF},  the consistency assumptions needed by these schemes. This type of scheme evolves two sets of degrees of freedom: point values that are chosen on the boundary of the elements that cover the computational domain, and the average of the solution in these elements. We show that the only thing to do is to get an entropy inequality for the average values, the point values degrees of freedom do not play any role. We construct a monolithic scheme which is bound preserving of \cite{BP_Pampa_VEM}, 
non oscillatory following \cite{OEPampa},
and entropy diminishing. The entropy condition is implemented in Tadmor's framework\cite{TadmorEntropy}, i.e. for the semi-discrete scheme only. The scheme is tested on the Kurganov-Popov-Petrova test case \cite{KPP} which is known to \yolo{be} sensitive to the satisfaction of an entropy inequality. We show that our entropy correction is effective: if we do not activate the bound-preserving nor the non oscillatory condition, we get the correct solution with some spurious wiggles, as expected. Though the development, implementation and tests are done with the triangle version of the scheme, the same method can be used for polygonal meshes, following \cite{BP_Pampa_VEM}.
\end{abstract}
\maketitle
%\listoftodos
%\input{notation.tex}
%\tableofcontents
%%%%%%%%%%%%%%%
\section{Introduction}
We are interested in the numerical discretisation of the hyperbolic system of conservation laws
\begin{equation}
\label{eq:1}
\begin{split}
\dpar{\bbu}{t}+\text{ div }\bbf(\bbu)&=0,\\
\bbu(\bbx,0)&=\bbu_0(\bbx),
\end{split} \qquad \bbx\in \R^d, ~t>0
\end{equation}
with the usual assumptions on the flux $\bbf$. The problem \eqref{eq:1} is defined for $\bbu\in \DD\subset \R^{\yolo{q}}$ where $\DD$ is the invariant domain.  We also assume the existence of an entropy $\eta$, i.e a strictly convex function defined on $\DD$, and an entropy flux $\bbg$ satisfying
$$\big (\nabla_\bbu\eta\big )^T\nabla_\bbu\bbf=\nabla_\bbu\bbg$$
such  the solutions of \eqref{eq:1} also satisfy 
\begin{equation}\label{eq:entropy}
\dpar{\eta(\bbu)}{t}+ \bbg(\bbu)\leq 0
\end{equation}
in the sense of distribution.

The canonical example is that of the Euler equations where
$$\bbu=\begin{pmatrix}\rho \\ \bbm\\ E\end{pmatrix} \quad \bbf=\begin{pmatrix} \bbm \\ \frac{\bbm\otimes \bbm}{\rho}+p\Id\\
\frac{(E+p)}{\rho}\bbm
\end{pmatrix}$$
where  $\rho$ is the density, $\bbm$ the momentum, $E$ the total energy. We define the internal energy
$$e=E-\frac{\bbm^2}{2\rho},$$
the pressure is assumed to be of the form $p=p(\rho, e)$, and the velocity $\bm\nu=\bbm/\rho$. The simplest model is that of a perfect gas,
$$p=(\gamma-1)e,$$ where $\gamma$ is a constant. We stick to this equation of state (EOS) in this paper, the generalisation to thermo-dynamically consistant EOS can be proceed with the same technique.
For the Euler equations, $\yolo{q}=2+d$, and the invariant domain is
$$\DD=\{\bbu=(\rho, \bbm, E)^T\text{ such that }\rho>0 \text{ and }e>0\}.$$
This set is strictly convex.

The format of this paper is as follows. First we describe the scheme in the case of a triangular mesh, with some emphasis on a blended scheme between a first order one and a high order one. Then we show a formal truncation error analysis, and recall the proof of a Lax-Wendroff like theorem. In passing, we refine the assumptions for the point value update in comparison with those given in \cite{Abgrall_AF}. The third part we show how to obtain an entropy satisfying flux. We show in particular that we only need to do something on the average value, not the point values. We demonstrate the effectiveness of the procedure on the Kurganov-Popov-Petrova test case which is known to be very sensitive to the satisfaction, or not, of an entropy condition. Some conclusions follow. In particular this method extends without modification to the scheme of \cite{BP_Pampa_VEM}.
%%%%%%%%%%%%%%%%%%%%%%%%%%%%%

\section{Numerical schemes}
%%%%%%%%%%%%%%%%%%%%%%%%%%%%%
%%%%%%%%%%%%%%%%%%%%%%%%%%%%%

We are interested in Active Flux-type discretisation. In this paper we specialise to schemes that are defined on triangular type meshes. The case of polygonal meshes, as covered in \cite{BP_Pampa_VEM}, will not be considered in this note to simplify the notations but exactly the same technique could be used.

In this section, after having described the scheme, we recall the consistency requirements that are needed to show a Lax Wendroff theorem, as well as a formal truncation error. 
\subsection{Description}
%%%%%%%%%%%%%%%%%%%%%%%%%%%%%

We consider a family $\mathcal{T}_h$ of simplex in $\R^d$, $d=2,3$, denoted generically by $K$. For a given $K$, $h_K$ is its diameter and $\rho_K$ is inner diameter. The maximum of the $h_K$ is denoted by $h$. Let $\gamma\in ][1,+\infty[$. We consider the family of triangulations such that for all $K\in \mathcal{T}_h$,  $\rho_k/h_K>\gamma$. These triangulations are regular in the sens of finite element. In the following, the parameter $\gamma$ will no longer appear, however when we talk about regular triangulation, we mean that we consider a value of $\gamma$, and all the triangulations satisfy this requirement for that value of $\gamma$.

For simplicity we stick to $d=2$, the generalisation to $d=3$ is immediate.  The triangulation  is assumed to be conformal, i.e. if $K$ and $K'$ are two polygons, then either $K\cap K'=\emptyset$ or $K\cap K'$ is a common face of $K$ and $K'$, or $K\cap K'$ reduces to a common vertex. The conformity assumption can be removed if we would have had a polygonal mesh. The set of faces/edges of $K$ is $\FF_K$.

The solution of \eqref{eq:1} is, \yolo{according to} \cite{AbgrallLinLiu2025}, 
in the following space
$$V_h=\bigg (\oplus_{K}V(K)\bigg )\cap C^0(\R^d)$$ with
$$V(K)=(\P^2(K)\oplus \R b\big )^{\yolo{q}}$$
where, if $\lambda_1,\lambda_2, \lambda_3$ are the barycentric coordinate of the triangle $K$, the bubble $b$ is
$$b=60\lambda_1\lambda_2\lambda_3.$$
The coefficient is chosen such that $\int_Kb\; d\bbx=\vert K\vert$. A function in $V_h(K)$ is completely described 
 by the following degrees of freedom:
\begin{itemize}
\item The values of $\bbu$ at the vertices of $K$
\item For any edge $e\in \FF_K$, the value of $\bbu$ at the center of $e$,
\item The average value of $\bbu$ on $K$.
\end{itemize}
We denote the point values degrees of freedom by $\bsigma$ (i.e. vertices and mid-points), the value of $\bbu$ at $\bsigma$ by $\bbu_\bsigma$. The average value is denoted by $\xbar{\bbu}_K$.
If $\{\varphi_\bsigma\}$ are the Lagrange basis functions with respect to the point value degrees of freedom (the classical $\P^2$ Lagrange functions), the basis functions $\{\psi_\bsigma\}$ are:
\begin{itemize}
\item For the vertex $\bsigma$
$$\psi_\bsigma=\varphi_\bsigma,$$
\item For the mid-point of the edges,
$$\psi_\bsigma=\varphi_\bsigma-\frac{1}{3}b.$$
\end{itemize}
If $\bbu\in V_h(K)$, then
$$\bbu_h=\sum_{\bsigma\in K}\bbu_\bsigma\psi_{\bsigma}+\xbar\bbu_K b.$$
Here we have chosen to consider the case of a quadratic approximation for the simplicity of the text only.

The semi-discrete scheme is:
\begin{subequations}\label{scheme}
\begin{itemize}
\item For the average values
\begin{equation}
\label{scheme:ave}
\vert K\vert \dfrac{d\xbar{\bbu}_K}{dt}+\oint_{\partial K} \hbbf_\bbn\; d\gamma=0
\end{equation}
where $\vert K\vert$ is the volume of $K$, $\bbn$ is the outward unit vector, $\vert e\vert$ is the length of the edge $e$ and 
\begin{equation}
\label{scheme:ave2}
\oint_{\partial K} \hbbf_\bbn\; d\gamma=\sum_{e\in \FF_K}\vert e\vert \hbbf_{\bbn_e}(\yolo{\xbar\bbu}_K,\yolo{\xbar\bbu}_{K_{e,-}},\yolo{\bbu_{\bsigma\in K}})
\end{equation}
and using a Gaussian quadrature formula, with weights $\omega_{quad}$ and quadrature points $\bbx_{quad}$,
\begin{equation}\label{scheme:ave:3}
\hbbf_{\bbn_e}(\yolo{\xbar\bbu}_K,\yolo{\xbar\bbu}_{K_{e,-}},\yolo{\bbu_{\bsigma\in K}})=\ell_e\bigg (\underbrace{ \sum_{quad}\omega_{quad}\;\bbf\big (\bbu_h(\bbx_{quad})\big)\cdot\bbn_e}_{\hbbf^{HO}_{\bbn_e}}\bigg )+(1-\ell_e)\hbbf_{\bbn_e}^{LO}( \yolo{\xbar\bbu}_K,\yolo{\xbar\bbu}_{K_{e,-}})
\end{equation}
where $K_{e,-}$ is the polygon on the other side of $e$. Often we will write the flux as $\hbbf_{\bbn}$ when there is no ambiguity.  The parameter $\ell_e$ is chosen as in \cite{BP_Pampa_VEM} or \cite{OEPampa}, or in order to satisfy an entropy inequality. Here, $\hbbf^{LO}_{\bbn_e}$ is Rusanov's  numerical flux,
$$\hbbf_{\bbn_e}^{LO}( \yolo{\xbar\bbu}_K,\yolo{\xbar\bbu}_{K_{e,-}})=\frac{1}{2}\bigg (\bbf(\yolo{\xbar\bbu}_K)\cdot\bbn_e+\bbf(\yolo{\xbar\bbu}_{K_{e,-}})\cdot\bbn_e\bigg)+{\alpha}(\yolo{\xbar\bbu}_K-\overline{\overline{\bbu}}_K)$$
with $$\yolo{\overline{\overline{\bbu}}_K=\frac{\xbar{\bbu}_K+\xbar{\bbu}_{K_{e,-}}}{2}}$$ and $\alpha$ \yolo{is} the largest speed $\alpha(\yolo{\xbar\bbu}_K, \yolo{\xbar\bbu}_{K_{e,-}}, \bbn_e)$ of the Riemann problem between the states $\yolo{\xbar\bbu}_K$ and $\yolo{\xbar\bbu}_{K_{e,-}}$ in the direction $\bbn_e$, see \cite{GuermondPopovFast} for a practical evaluation. 

We will discuss later the choice of $\ell_e$ in the text.
\item For the point values
\begin{equation}
\label{scheme:pt}
\dfrac{d\bbu_\bsigma}{dt}+\sum_{K, \bsigma\in K} \bPhi_\bsigma^K(\bbu)=0
\end{equation}

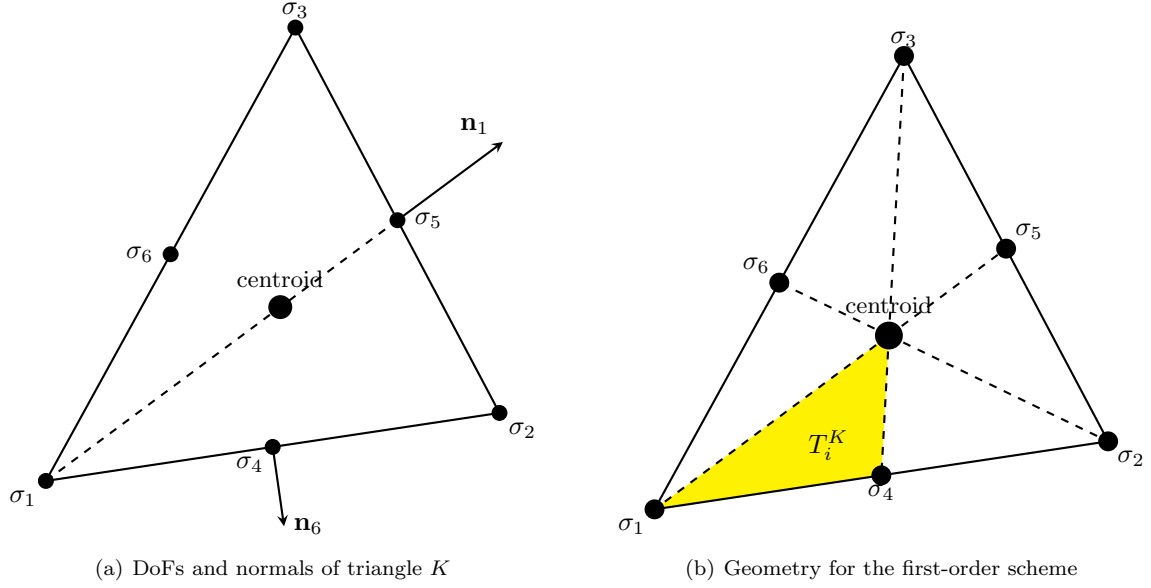
\begin{figure}[ht!]
\begin{center}
\subfigure[DoFs and normals of triangle $K$]{\scalebox{1}{ \begin{tikzpicture}[>=stealth, scale=1.5]

    % --- Coordonnées des sommets du triangle ---
    \coordinate (S1) at (0,0);
    \coordinate (S2) at (4,0.6);
    \coordinate (S3) at (2.2,4);

    % --- Calcul des milieux des arêtes ---
    \coordinate (S4) at ($(S1)!0.5!(S2)$);
    \coordinate (S5) at ($(S2)!0.5!(S3)$);
    \coordinate (S6) at ($(S3)!0.5!(S1)$);

    % --- Calcul du centroïde (centre de gravité) ---
    \coordinate (C) at (barycentric cs:S1=1,S2=1,S3=1);

    % --- Dessin de la géométrie ---
    \draw[thick] (S1) -- (S2) -- (S3) -- cycle; % Triangle
    \draw[dashed, thick] (S1) -- (S5);         % Trait pointillé entre sigma_1 et sigma_5

    % --- Tracé et étiquetage des points (nœuds) ---
    \fill (S1) circle (2pt) node[below left] {$\sigma_1$};
    \fill (S2) circle (2pt) node[below right] {$\sigma_2$};
    \fill (S3) circle (2pt) node[above] {$\sigma_3$};
    \fill (S4) circle (2pt) node[below left] {$\sigma_4$};
    \fill (S5) circle (2pt) node[right=3pt] {$\sigma_5$};
    \fill (S6) circle (2pt) node[left=3pt] {$\sigma_6$};
    \fill (C) circle (3pt) node[above=4pt] {\small centroid};

    % --- Vecteur -n1 (part de sigma_5 et pointe vers l'extérieur) ---
\draw[->, thick] (S5) -- ($ (S5) + 0.3*(S5) - 0.3*(S1) $) node[midway, above right, yshift=14pt] {$\mathbf{n}_{1}$};

    % --- Vecteur n6 (partant de sigma_4, pointant vers l'extérieur) ---
    \draw[->, thick] (S4) -- +(0.1, -0.7) node[right] {$\mathbf{n}_6$};

\end{tikzpicture}}}\hspace{0.5cm}
\subfigure[Geometry for the first-order scheme]{\scalebox{1}{ \begin{tikzpicture}[>=stealth, scale=1.5]

    % --- Coordonnées des sommets du triangle ---
    \coordinate (S1) at (0,0);
    \coordinate (S2) at (4,0.6);
    \coordinate (S3) at (2.2,4);

    % --- Calcul des milieux des arêtes ---
    \coordinate (S4) at ($(S1)!0.5!(S2)$);
    \coordinate (S5) at ($(S2)!0.5!(S3)$);
    \coordinate (S6) at ($(S3)!0.5!(S1)$);

    % --- Calcul du centroïde (centre de gravité) ---
    \coordinate (C) at (barycentric cs:S1=1,S2=1,S3=1);

    % --- Remplissage de la zone jaune (sous-triangle T_i^K) ---
    \fill[yellow] (S1) -- (S4) -- (C) -- cycle;

    % --- Dessin de la géométrie principale ---
    \draw[thick] (S1) -- (S2) -- (S3) -- cycle; % Triangle extérieur

    % --- Lignes pointillées internes ---
    \draw[dashed, thick] (C) -- (S1);
    \draw[dashed, thick] (C) -- (S2);
    \draw[dashed, thick] (C) -- (S3);
    \draw[dashed, thick] (C) -- (S4);
    \draw[dashed, thick] (C) -- (S5);
    \draw[dashed, thick] (C) -- (S6);

    % --- Tracé et étiquetage des nœuds ---
    \fill (S1) circle (2.5pt) node[below left] {$\sigma_1$};
    \fill (S2) circle (2.5pt) node[below right] {$\sigma_2$};
    \fill (S3) circle (2.5pt) node[above] {$\sigma_3$};
    \fill (S4) circle (2.5pt) node[below] {$\sigma_4$};
    \fill (S5) circle (2.5pt) node[above right] {$\sigma_5$};
    \fill (S6) circle (2.5pt) node[above left] {$\sigma_6$};
    
    % Centroïde (forme hexagonale simulée par un cercle épais ou nœud)
    \fill (C) circle (3.5pt) node[above=5pt] {\small centroid};

    % --- Étiquette de la zone jaune ---
    \node at ($(S1)!0.5!(C) + (0.5,-0.2)$) {$T_i^K$};

\end{tikzpicture}}}
\caption{\label{fig:1}Geometry notations.} %(a): DoFs and normals of triangle $K$, (b): Geometry for the first-order scheme. 
\end{center}
\end{figure}

In \eqref{scheme:pt}, the residual $\bbPhi_\bsigma^K$ is defined as
\begin{equation}
\label{scheme:pt:blend}\bbPhi_\bsigma^K=\ell_\bsigma \bbPhi_\bsigma^{HO,K}+(1-\ell_\bsigma)\bbPhi_{\bsigma}^{LO,K}
\end{equation}
The parameter $\ell_\sigma$ is defined as in \cite{BP_Pampa_VEM} or \cite{OEPampa}, or to satisfy an entropy inequality. This will be discussed in details later in the text.

We need to specify the high order and low order residuals in \eqref{scheme:pt:blend}. The high order residual $\bbPhi_\bsigma^{HO,K}$ is defined as:
\begin{equation}\label{scheme:pt:HO}
\bbPhi_\bsigma^{HO,K}=\bbN_\bsigma \big (\big ( \bbK_\bsigma\big )^++\varepsilon_\bsigma^K\Id\big ) \bbK_\bsigma \cdot\nabla\bbu_h\yolo{(\bbx_\bsigma)}.\end{equation}
Here,  we have set $\bbK_\bsigma=\nabla\bbf(\bbu_\bsigma)\cdot \bbn_\bsigma^K$, $\varepsilon_\bsigma^K>0$. The normal $\bbn_\bsigma^K$ is defined as the sum of the outward normals of the two segments that share $\bsigma$, see Figure \ref{fig:1}-(a).  Since $\bbK_\bsigma$ is diagonalisable in $\R$, we can take its positive part. The parameter $\varepsilon_\bsigma^K$ is set to $\vert K\vert/2$\footnote{but what really matters is that it is strictly positive.}, and 
$$\bbN_\bsigma^{-1}=\sum_{K, \bsigma\in K}\big (\big ( \bbK_\bsigma\big )^++\varepsilon_\bsigma^K\Id\big )$$ which is invertible if \eqref{eq:1} is endowed with an entropy, see \cite{BP_Pampa_VEM,OEPampa} for details.

For the low order residual we use the notations of Figure \ref{fig:1}-(b). The average value is identified to the value of $\bbu$ at the centroid displayed as $\bsigma_7$.  From $\{\bsigma_1, \ldots, \bsigma_7\}$ we construct $6$ sub-triangles denoted by $T^K_i$, $i=1, \ldots, 6$. One boundary vertex belongs to two sub-triangles. We define
$\bbPhi_{K,\sigma}^{\text{LO}}$ as
\begin{equation}\label{scheme:pt:LO1}
\bbPhi_{K,\sigma}^{\text{LO}}=\frac{1}{\vert C_\sigma\vert }\sum\limits_{T^K_i, \bsigma\in T^K_i}\widehat{\bbPhi}_{\bsigma,T^K_i}^{LO}(\bbu),
\end{equation}
where $$\vert C_\sigma\vert=\sum\limits_{K, \sigma\in K}\sum\limits_{T^K_i, \bsigma\in T^K_i} \vert T^K_i\vert, \quad \vert T_i^K\vert=\frac{\vert K\vert}{6},$$ is the measure of the dual control volume and
\begin{equation}\label{scheme:pt:LO2}
\begin{split}
   \widehat{\bbPhi}_{\bsigma,T^K_i}^{LO}&=\frac{1}{3}\sum_{\bsigma_l}\yolo{\frac{1}{2}}\bbf(\bbu_{\bsigma_l})\cdot \bbn_{\bsigma_l}^{T_i^K}+ \alpha_{T^K_i}\big (\bbu_\sigma-\overline{\overline{\bbu}}_{T^K_i}\big )\\
  \overline{\overline{\bbu}}_{T^K_i}&=\frac{1}{3}\sum_{\bsigma_l\in T^K_i}\bbu_{\bsigma_l}.
 \end{split}
\end{equation}
The vector $\bbn_{\bsigma_l}^{T^K_i}$ in \eqref{scheme:pt:LO2}  is the scaled outward normal of the edge of $T_i^K$ opposite to $\bsigma_l$. 
We can also write \eqref{scheme:pt:LO2} as
\begin{equation}
\label{Theta}\widehat{\bbPhi}_{\bsigma,T^K_i}^{LO}=\frac{1}{3}
\sum_{\bsigma_l\in T_i^K} \bigg ( \yolo{\frac{1}{2}}\big (\bbf(\bbu_{\bsigma_l})-{\bbf(\bbu_{\bsigma})}\big )
\cdot \bbn_{\bsigma_l}^{T^K_i}
  +\alpha_{T^K_i}\big (\bbu_\sigma-\bbu_{\bsigma_l}\big )\bigg ).
  \end{equation}
 \end{itemize}
\end{subequations}

\bigskip
The relations \eqref{scheme:ave} and \eqref{scheme:pt} are discretised in time by using a SPP Runge-Kutta scheme. This enables to approximate the solution of \eqref{eq:1} in $\R^d$ at any $t_n$ defined by $t_0=0$ and $t_{l+1}=t_{l}+\Delta t_l$.  Using SSP Runge-Kutta temporal discretisation allows to concentrate only on the Euler forward discretisation of \eqref{scheme:ave} and \eqref{scheme:pt}.

Next, we  show a formal truncation error analysis, and using this algebra, we can show a Lax-Wendroff like theorem. This allows to give precise structure conditions on the scheme.
%%%%%%%%%%%%%%%%%%
\subsection{Truncation error analysis}\label{sec:Truncation_error}
We consider the following hyperbolic problem  in $\R^d$
\begin{subequations}
\label{appendix:eq:1}
\begin{equation}
\label{appendix:eq:1:1}
\dpar{\bbu}{t}+\text{div }{\bbf(\bbu)}=0, \qquad \bbx\in \R^d
\end{equation}
with the initial condition 
\begin{equation}
\label{appendix:eq:1:2}
\bbu(x,0)=\bbu_0(x), \qquad \bbx\in \R^d.
\end{equation}
\end{subequations}

\yolo{We consider the problem in \eqref{eq:1} with $\bbu\in \mathcal{D}\subset \R^q$.} 
We want to establish a truncation error formula, i.e. some approximation of, for any test function, 
$$\mathcal{E}(\varphi, \bbu_h)=\int_{\R^d}\varphi(\bbx,t)\bigg ( \dpar{\bbu_{h}}{t}+\text{div }{\bbf(\bbu_{ h})}\bigg )\; d\bbx,$$
where $\bbu_{{h}}$ is some interpolant of the exact solution, and see a formal condition of the form
$$\Vert \mathcal{E}(\varphi, \bbu_h)\Vert \leq C(\Vert \bbu_h\Vert)\; \Vert \varphi\Vert\times h^\alpha$$
where $h$ is a measurement of the mesh size and $\Vert~.~\Vert$ some norm on $C^1_0(\R^d)$.

We assume $d=2$ to fix ideas, the case $d=3$ would be done the same. We consider a triangulation of $\R^2$, made of triangles. The triangles are denoted generically by $K$. The vertices of the triangles are denoted by $\bba_j$, and the degrees of freedom (hence including these vertices) by $\bsigma_l$. The centroid of $K$ is $\bbx_K$. We look at the quadratic case, the extension to higher degree is natural.

In the particular case of the approximation we consider, we know that
\begin{subequations}\label{appendix:simpson}
\begin{equation}
\label{appendix:simpson:1}
\frac{1}{\vert K\vert}\int_K \bbu_{ h}\; d\bbx=\xbar\bbu_K=
\sum_{j=1}^6 \beta_j\bbu_{ h}(\bsigma_j)+\beta_7 \bbu_{ h}(\bbx_K)
\end{equation}
where $\beta_j=\yolo{\tfrac{1}{20}}$ for $j=1,2,3$, $\yolo{\beta_j=\tfrac{2}{15}}$ for $j=4,5,6$ and $\yolo{\beta_7=\tfrac{9}{20}>0}$. We also have
$$\sum_{j=1}^7 \beta_j=1.$$ This is  a Simpson-like formula from which we get
\begin{equation}\label{appendix:simpson:2}
\bbu_{ h}(\bbx_K)=\theta_7\xbar \bbu_K +\sum_{j=1}^6\theta_j \bbu_{ h}(\bsigma_j)
\end{equation}
\end{subequations}
with $\theta_j=-\tfrac{1}{9}$ for $j=1,2,3$, $\theta_j=-\tfrac{8}{27}$, $j=4,5,6$ and,  $\theta_7=\tfrac{20}{9}$. We also have $$\sum_{j=1}^7 \theta_j=1.$$

Taking $\varphi\in C^1(\R^d)$, we can write
$$
\int_K\varphi \bbv\;  d\bbx=\vert K\vert \bigg (
 \sum_{j=1}^6\beta_j\varphi(\bsigma_j) \bbv(\bsigma_j)+\beta_7 \varphi(\bbx_K)\bbv(\bbx_K)\bigg )+\vert K\vert\;{O(h^3)}
$$
since the formula is exact for quadratic functions.
In the following, we set
$$\II(\varphi, \bbv, K)=
 \sum_{j=1}^6\beta_j\varphi(\bsigma_j) \bbv(\bsigma_j)+\beta_7\varphi(\bbx_K)\bbv(\bbx_K).$$
 \begin{remark}[About the Simpson-like formula]
 The formula \eqref{appendix:simpson:1} with positive weights can be shown to be true for cubic approximation and VEM (quadratic) approximation, see \cite{pampa_dg} for details.
 \end{remark}
 The semi-discrete scheme for quadratics writes
 \begin{subequations}
\label{appendix:scheme}
\begin{equation}
\label{appendix:scheme:1}
 \dfrac{ d\xbar \bbu_{K}}{ dt}+  \frac{1}{\vert K\vert }\oint_{\partial K}\hbbf_\bbn\;  d\ell  =0
\end{equation}
and for $\bbu_\bsigma$, we assume a semi-discrete scheme of the following form:
\begin{equation}
\label{appendix:scheme:3}
 \dfrac{ d\bbu_\bsigma}{ dt}+ \sum_{K,\bsigma\in K}\bomega_{K,\bsigma} \bbPhi_{K,\bsigma}(\bbu_{ h})=0,
\end{equation} 
The scalar/matrices $\bomega_{K,\bsigma}(\bbu)$ satisfy the assumptions of \eqref{bomega}.
\end{subequations}
\begin{remark}
In \eqref{appendix:scheme:1} we have used the  numerical flux $\hbbf_\bbn$ in the direction $\bbn$ instead of the continuous one, $\bbf(\bbu_{ h})\cdot \bbn$ to account for the possibility of having  a flux that is blended between the continuous one and a first order one. The same will be done for the residuals $\bbPhi_\bsigma^K$, we do not need to write it explicitly.
\end{remark}
Using \eqref{appendix:simpson:1} we see that
$$
\II(\varphi, \bbv, K)=\varphi(\bbx_K) \xbar \bbv_K+ \sum_{j=1}^6\beta_j\big ( \varphi(\bsigma_j)-\varphi(\bbx_K)\big )\bbv(\bsigma_j).$$
The matrices/scalars $\bomega_{K,\bsigma}$ satisfy
\begin{equation}\label{bomega}\sum_{K, \bsigma\in K}\bomega_{K,\bsigma}=\Id \text{ and }\Vert \bomega_{K,\bsigma}\Vert \leq C
\end{equation}
for some $C$ independent of $K$, of the degree of freedom in $K$ and the class of mesh, it may only  depend  on some norm of $\bbu_{ h}$, for example the $L^\infty$ norm.

We use this on the scheme and get for each $K$
$$\II(\varphi, \dfrac{ d\bbu_{ h}}{ dt}, K)=\varphi(\bbx_K) \dfrac{ d\xbar\bbu_K}{ dt}+ \sum_{j=1}^6\beta_j\big ( \varphi(\bsigma_j)-\varphi(\bbx_K)\big )\dfrac{ d\bbu_{ h}}{ dt}(\bsigma_j),$$
and then insert the schemes to obtain
\begin{equation*}
    \begin{aligned}
\II(\varphi, \dfrac{ d\bbu_{ h}}{ dt}, K)&=
-\varphi(\bbx_K) \yolo{\frac{1}{\vert K\vert}}\oint_{\partial K}\hbbf_\bbn\;  d\ell\\
&-\sum_{j=1}^6\beta_j\big ( \varphi(\bsigma_j)-\varphi(\bbx_K)\big )\bigg (\sum_{K, \bsigma_j\in K}\bomega_{K,\bsigma_j}\bbPhi_{K,\bsigma_j}(\bbu_{ h})\bigg ).
    \end{aligned}
\end{equation*}
So that, using \eqref{appendix:scheme},  and since \eqref{appendix:scheme:3}, 
we get
\begin{equation}\label{appendix:schema:bis}
\begin{split}
\sum_K  \varphi(\bbx_K)\bigg (& \vert K\vert\dfrac{ d\xbar\bbu_K}{ dt}+ \oint_{\partial K}\hbbf_\bbn\;  d\ell\bigg )\\+
&\sum_K \vert K\vert \sum_{j=1}^6 \beta_j\big ( \varphi(\bsigma_j)-\varphi(\bbx_K)\big )\bomega_{K,\bsigma_j}\bigg (\dfrac{ d\bbu_{ h}}{ dt}(\bsigma_j)+
\bbPhi_{K,\bsigma_j}(\bbu_{ h})\bigg)=0
\end{split}
\end{equation}
The relation \eqref{appendix:schema:bis} is a \emph{rewriting} of the original scheme, tested on any test function $\varphi$.

The (weak) truncation error of the scheme is
\begin{equation}
\label{appendix:weak:TE}
\begin{split}
\varepsilon(\varphi, \bbu^{{ex}}):=
&\sum_K \vert K\vert \varphi(\bbx_K)\bigg ( \dfrac{ d\xbar{(\pi \bbu^{{ex}})}}{ dt}+\frac{\oint_{\partial K}\hbbf_\bbn\; d\ell}{\vert K\vert }\bigg ) \\
&+
\sum_K\vert K\vert \sum_{j=1}^6 {\beta_j}  \big ( \varphi(\bsigma_j)-\varphi(\bbx_K)\big ){\bomega_{K,\bsigma_j}}\bigg ( \dfrac{ d\yolo{\pi\bbu^{{ex}}}}{ dt}(\bsigma_j)+\bbPhi_{K,\bsigma_j}(\pi\bbu^{{ex}})\bigg )
\end{split}
\end{equation}
where 
$$\pi\bbu^{{ex}}=\sum_{j=1}^6 \bbu^{{ex}}(\bsigma_j)\yolo{\psi}_{\bsigma_j}+\xbar \bbu^{{ex}}\yolo{b}.$$
From \eqref{appendix:weak:TE}, we see that to get $\varepsilon(\varphi, \bbu^{{ex}})=O(h^3)$ which is the expected behavior, it is enough to have
$$ \dfrac{ d\xbar{(\pi \bbu^{{ex}})}}{ dt}+\frac{\oint_{\partial K}\hbbf_\bbn\; d\ell}{\vert K\vert }=O(h^3)
\text{ and }
\dfrac{ d\yolo{\pi\bbu^{{ex}}}}{ dt}(\bsigma_j)+\bbPhi_{K,\bsigma_j}(\pi\bbu^{{ex}})=O(h^2)$$ because $\varphi(\bsigma_j)-\varphi(\bbx_K)=O(h)$.

Indeed, we first have that, writing $e=\pi\bbu^{{ex}}-\bbu^{{ex}}$, 
\begin{equation*}
\begin{split}\varepsilon(\varphi, \bbu^{{ex}})&=
\sum_K\varphi(\bbx_K) \bigg (\vert K\vert  \dfrac{ d\xbar e_K}{ dt}+\oint_{\partial K}\big (\hbbf_\bbn-\bbf(\yolo{\bbu^{ {ex}}})\cdot \bbn\big )\; d\ell\big ) \\
&\hspace{-3.5em}+
\sum_K\vert K\vert \sum_{j=1}^6 {\beta_j\bomega_{K,\bsigma_j}}  \big ( \varphi(\bsigma_j)-\varphi(\bbx_K)\big )\bigg (\underbrace{ \dfrac{ de_{\bsigma_j}}{ dt}+\big (\bbPhi_{K,\bsigma_j}\yolo{(\pi\bbu^{ex})}-\bbPhi_{K,\bsigma_j}\yolo{(\bbu^{ {ex}})}\big )}_{(II)}\yolo{+\dfrac{ d\bbu^{ex}}{ dt}(\bsigma_j)+\bbPhi_{K,\bsigma_j}(\bbu^{ {ex}})}\bigg ).
\end{split}
\end{equation*}
Now by taking \yolo{$\hbbf_\bbn$ and} $\bbPhi_{K,\sigma_j}$ as in \eqref{scheme:ave}-\eqref{scheme:pt} with $\ell_e=\ell_\bsigma=1$, 
since $\Vert e\Vert_{\infty}\leq C\; h^3$ where the constant $C$ only depends on the $L^\infty$ norm of $\bbu_{ {ex}}$ and its derivatives, we see that 
$\frac{ d\bar{e}_K}{ dt}$ is $O(h^3)$, while, if $\hbbf_\bbn-\bbf(\bbu^{ex})\cdot \bbn=O(h^3)$, we get
\begin{equation*}
\begin{split}\sum_K \varphi(\bbx_K)&\oint_{\partial K}\big (\hbbf_\bbn-\bbf(\yolo{\bbu^{ {ex}}})\cdot \bbn\big )\; d\ell\\
&=\yolo{\sum_K}\sum_{\text{edges} }\big (\varphi(\bbx_K)-\varphi(\bbx_{K'})\big )\int_e\big (\hbbf_\bbn-\bbf(\yolo{\bbu^{ {ex}}})\cdot \bbn\big )\; d\ell\yolo{=\sum_K\sum_{\text{edges}}\vert e\vert O(h^4)}=\yolo{\sum_K}\vert K\vert O(h^{3})
\end{split}
\end{equation*}
since $\varphi(\bbx_K)-\varphi(\bbx_{K'})=O(h)$, so that $\vert e\vert (\varphi(\bbx_K)-\varphi(\bbx_{K'}) =O(h^2)=\vert K\vert O(1)$ for a regular mesh. This leads to  
$$\sum_K\varphi(\bbx_K) \bigg (\vert K\vert  \dfrac{ d\xbar e_K}{ dt}+\oint_{\partial K}\big (\hbbf_\bbn-\bbf(\yolo{\bbu^{ {ex}}})\cdot \bbn\big )\; d\ell\bigg ) =O(h^3).$$
The term \yolo{$\dfrac{ d\bbu^{ex}}{ dt}(\bsigma_j)+\bbPhi_{K,\bsigma_j}(\bbu^{ {ex}})$ is $O(h^2)$} and $(II)$ is also only $O(h^2)$ because $\bbPhi_{K,\sigma_j}(\pi\bbu^{ex})-\bbPhi_{K,\sigma_j}(\bbu^{ {ex}})=O(h^2)$ but since $\varphi(\bsigma_j)-\varphi(\bbx_K)=O(h)$, we see that in the end
$$\sum_K\vert K\vert \sum_{j=1}^6 {\beta_j\bomega_{K,\bsigma_j}}  \big ( \varphi(\bsigma_j)-\varphi(\bbx_K)\big )\bigg (\dfrac{ de_{\bsigma_j}}{ dt}+\big (\bbPhi_{K,\bsigma_j}{(\pi\bbu^{ex})}-\bbPhi_{K,\bsigma_j}{(\bbu^{ {ex}})}\big )+\dfrac{ d\bbu^{ex}}{ dt}(\bsigma_j)+\bbPhi_{K,\bsigma_j}(\bbu^{ {ex}})\bigg )=O(h^3),$$
so that $\varepsilon(\varphi, \bbu_{{ex}})=O(h^3)$ for regular meshes.

\begin{remark} This is not specific to two dimensional problems, because we have used that $\vert K\vert =O(h^2)$ and $\vert e\vert=O(h)$, that is in general $\vert K\vert =O(h^d)$ and $\vert e\vert =O(h^{d-1})$ for a regular mesh.
\end{remark}
%%%%%%%%%%%%%

\subsection{Lax-Wendroff like theorem}
Here we consider a fully discrete version of the scheme, where, for simplicity, the temporal scheme is the Euler forward scheme \yolo{from $t_n$ to $t_{n+1}$}.
The quantity of interest is now the fully discrete analog of \eqref{appendix:weak:TE}:
\begin{equation}
\begin{split}
\sum_{n=0}^\infty \Delta t \Bigg [ \sum_K  \varphi(\bbx_K, t_n)\bigg (& \vert K\vert\dfrac{\xbar\bbu_K^{n+1}-\xbar\bbu_K^n}{\Delta t}+ \oint_{\partial K}\hbbf_\bbn\;  d\ell\bigg )\\+
&\sum_K \vert K\vert \sum_{j=1}^6 \beta_j\big ( \varphi(\bsigma_j, t_n)-\varphi(\bbx_K,t_n)\big )\bomega^n_{K,\bsigma_j}\bigg (\dfrac{\bbu_{\bsigma_j}^{n+1}-\bbu_{\bsigma_j}^n}{\Delta t}+
\bbPhi_{K,\bsigma_j}(\bbu_h^n)\bigg)\Bigg ]=0
\end{split}
\end{equation}
This is the sum of two terms. The first one is, choosing one orientation of the edges/faces of the mesh, writes as
\begin{equation*}
\begin{split}
I&=\sum_{n=0}^\infty \Delta t\sum_K  \varphi(\bbx_K, t_n)\bigg ( \vert K\vert\dfrac{\xbar\bbu_K^{n+1}-\xbar\bbu_K^n}{\Delta t}+ \oint_{\partial K}\hbbf_\bbn\;  d\ell\bigg )\\
&=\yolo{-}\sum_{n=1}^\infty \Delta t\sum_K\vert K\vert \frac{\varphi(\bbx_K, t_{n+1})-\varphi(\bbx_K, t_n)}{\Delta t}\xbar\bbu_K^{n+1}\\
&\qquad+\sum_{\yolo{n=0}}^\infty\yolo{\Delta t} \sum_{e\subset\text{ edge}}\big ( \varphi(\bbx_K, t_n)-\varphi(\bbx_{K_{e,-}}, t_n\big ) \int_e\hbbf_\bbn\;  d\ell\\
& \qquad\qquad -\sum_{K} \yolo{\vert K\vert}\varphi(\bbx_K, 0)\xbar\bbu_K^0
\end{split}
\end{equation*}
The second one is
\begin{equation*}
\begin{split}
II&=\sum_{n=0}^\infty \Delta t\sum_K \vert K\vert \sum_{j=1}^6 \beta_j\big ( \varphi(\bsigma_j, t_n)-\varphi(\bbx_K,t_n)\big )\bomega^n_{K,\bsigma_j}\bigg (\dfrac{\bbu_{\bsigma_j}^{n+1}-\bbu_{\bsigma_j}^n}{\Delta t}+
\bbPhi_{K,\bsigma_j}(\bbu_h^n)\bigg)
\end{split}
\end{equation*}
We have the following result:
\begin{proposition}\label{prop:LxW:AF}
We consider the sequence defined by the scheme and assume that:
\begin{enumerate}
\item The numerical flux $\hbbf$ is consistant and Lipschitz continuous,
\item The residuals $\bbPhi_{K,\bsigma}$ satisfy
\begin{equation}\label{ConsistantceResAF}\Vert \bbPhi_{K,\bsigma}\Vert \leq C\; \frac{\vert \partial K\vert}{\vert K\vert }\bigg (\sum_{\bsigma, \bsigma'\in K}\Vert \bbu_\bsigma-\bbu_{\bsigma'}\Vert +\sum_{\bsigma\in K}\Vert \bbu_\bsigma-\xbar{\bbu}_K\Vert \bigg )\end{equation}
where $C$ is independent of the mesh and $\bbu$,
\item The family of scalar/matrices $\bomega_{K,\bsigma}$ is bounded independently of the mesh by a constant that depends only on $\Vert \bbu\Vert_{L^\infty(\R^d\times \R^+)}$,
\item The sequence $\{ \bbu_\bsigma^n\}$ and $\{\xbar{\bbu}_K^n\}$ are bounded in $L^\infty$,
\item There exists a function of $\bbv\in L^2(\R^d\times \R+)$ such that the sequence $\bbu_\Delta$ defined as
$$\bbu_\Delta(\bbx, s)=\sum_{\bsigma}\bbu_\bsigma^n\varphi_\bsigma+\sum_K\xbar{\bbu}_K^n\yolo{b}, \qquad s\in [t_n, t_{n+1}[$$
converges to $\bbv$ in $L^2$.
\item The sequence of meshes is regular and for a mesh $\TT_h$, the  maximum of the diameters of the elements of $\TT_h$, denoted by $h$, tends to $0$. In other words, the ratio of the inner circles (sphere) and outer circles (sphere) over all the mesh elements is bounded from bellow by $\alpha>0$ independent of the mesh, and $h$ tends to $0$.
\end{enumerate}
Then $\bbv$ is a weak solution of the problem \eqref{appendix:scheme}.
\end{proposition}
\begin{proof}
Using classical arguments, for example see \cite{raviart}, we see that $(I)$ converges to
$$-\int_{\R^d}\varphi(\bbx,0) \bbu_0(\bbx)\; d\bbx\yolo{-}\int_{0}^{+\infty}\yolo{\int_{\R^d}}\dpar{\varphi}{t}(\bbx,s)\bbu(\bbx,s)+\nabla\varphi(\bbx,s)\cdot \bbf(\bbu)\; d\bbx\;\yolo{dt}.$$
The only thing to do is to show that $(II)$ will tend to $0$.
This is the purpose of the next sequence of lemma.
\end{proof}
\begin{lemma}\label{zelem}
    Under the conditions of proposition \ref{prop:LxW:AF}, we have, for all $\varphi\in C_0^1(\R^d\times \R^+)$
    $$\lim\limits_{\Delta t\rightarrow 0, h\rightarrow 0}\sum_{n=0}^\infty\sum_{K}\Delta t \vert K\vert \sum_{j=1}^6 \beta_j\big ( \varphi(\bsigma_j, t_n)-\varphi(\bbx_K,t_n)\big )\bomega^n_{K,\bsigma_j}\bigg (\dfrac{\bbu_{\bsigma_j}^{n+1}-\bbu_{\bsigma_j}^n}{\Delta t}+
\bbPhi_{K,\bsigma_j}(\bbu_h^n)\bigg)=0$$
\end{lemma}
which itself needs the following lemma inspired by \cite{kroner}:
\begin{lemma} \label{sec:rd:lemme1}
Let $T>0$, $N$ be the integer part of $T/\Delta t$ and $Q\subset \R^2$ bounded. We assume a family of regular triangulations $\{\mathcal{T}_h\}$.
We consider  a sequence $\bbu_{h}$, for $h\rightarrow 0$ 
defined by: for $t\in [t_n, t_{n+1}[$,\ and $\bbx\in K\in \mathcal{T}_h$,
$$\bbu_h(\bbx,s)=\sum_{\bsigma\in K}\bbu^n_{\bsigma}\varphi_\bsigma(\bbx)+\xbar{\bbu}_K^n b (\bbx).$$ We also assume that there exists $\bbu\in L_{loc}^2(Q\times [0,T])$ and a constant $C$ independent of $h$ and $\bbu$ such that
$$\sup_h\sup_{\bbx,t}\Vert \bbu_h(\bbx,t)\Vert\leq C, \text{ and }\lim_{h}\Vert \bbu_h-\bbu\Vert=0.$$
Then
$$\lim_{h\rightarrow 0}\bigg (\sum_{n=0}^N\Delta t\sum_{K\subset Q}\sum_{\bsigma\in K}\bigg [ \Vert \bbu_\bsigma^{n+1}-\bbu_\bsigma^n\Vert +\Vert \bbu_{\bsigma}^n-\xbar{\bbu}_K^n\Vert\bigg ]\bigg )=0$$
\end{lemma}
Its proof is postponed in the appendix \ref{sec:appendix}.

\begin{proof}[Proof of lemma \ref{zelem}.]
We have
$$\sum_{n=0}^\infty\sum_{K}\Delta t \vert K\vert \sum_{j=1}^6 \beta_j\big ( \varphi(\bsigma_j, t_n)-\varphi(\bbx_K,t_n)\big )\bomega^n_{K,\bsigma_j}\bigg (\dfrac{\bbu_{\bsigma_j}^{n+1}-\bbu_{\bsigma_j}^n}{\Delta t}+
\bbPhi_{K,\bsigma_j}(\bbu_h^n)\bigg)=A+B$$
with 
$$A=\sum_{n=0}^\infty \yolo{\sum_K}\vert K\vert \sum_{j=1}^6 \beta_j\big ( \varphi(\bsigma_j, t_n)-\varphi(\bbx_K,t_n)\big )\bomega^n_{K,\sigma_j}\big (\bbu_{\bsigma_j}^{n+1}-\bbu_{\bsigma_j}^n\big )$$
and
$$B=\sum_{n=0}^\infty \yolo{\sum_K}\Delta t\vert K\vert \sum_{j=1}^6 \beta_j\big ( \varphi(\sigma_j, t_n)-\varphi(\bbx_K,t_n)\big )\bomega^n_{K,\sigma_j}\bbPhi_{K,\bsigma_j}(\bbu_h^n)$$
We first have
$$\Vert A\Vert \leq C\sum_{n=0}^\infty \yolo{\sum_K}h\vert K\vert \sum_{\bsigma\in K}\Vert\bbu_{\bsigma}^{n+1}-\bbu_\bsigma^n\Vert $$
where $C$ is a bound independent of $\bbu$ and the mesh, but depending on $\Vert \nabla \varphi\Vert$. Using lemma \ref{sec:rd:lemme1}, we get that $A\rightarrow 0$ when $h\rightarrow 0$.  We also have
\begin{equation*}
    \begin{split}
        \Vert B\Vert& \leq C\sum_{n=0}^\infty \yolo{\sum_K}h \Delta t\vert K\vert \frac{\vert \partial K\vert}{\vert K\vert}\bigg (\sum_{\bsigma,\bsigma'\in K}\Vert \bbu_\bsigma^n-\bbu_{\bsigma'}^n\Vert+\sum_{\bsigma\in K}\Vert \bbu_\bsigma^n-\xbar{\bbu}_K^n\Vert \bigg ) \\
        &\leq C\sum_{n=0}^\infty \yolo{\sum_K}h \Delta t \vert \partial K\vert\bigg (\sum_{\bsigma,\bsigma'\in K}\Vert \bbu_\bsigma^n-\bbu_{\bsigma'}^n\Vert +\sum_{\bsigma\in K}\Vert \bbu_\bsigma^n-\xbar{\bbu}_K^n\Vert \bigg )\\
        & \leq C\sum_{n=0}^\infty\Delta t\yolo{\sum_K}\vert K\vert \bigg (\sum_{\bsigma,\bsigma'\in K}\Vert \bbu_\bsigma^n-\bbu_{\bsigma'}^n\Vert +\sum_{\bsigma\in K}\Vert \bbu_\bsigma^n-\xbar{\bbu}_K^n\Vert \bigg )
        \end{split}
        \end{equation*}
        which also converges to $0$ when $h\rightarrow 0$ from Lemma \ref{sec:rd:lemme1}. This ends the proof.
\end{proof}

\begin{remark}
    We see that the consistency requirement \eqref{ConsistantceResAF} on the point values is quite weak. This can be observed in the numerical simulation where with a wrong implementation of the point residuals one can get results that look good \ldots However, we can only get the correct convergence error only with a correct implementation. Hence, this means that the point value update is mostly playing a role for getting accurate results. This fact will be further exploited to design entropy satisfying schemes.
\end{remark}

\section{Extension to entropy}
Let us assume that the semi discrete scheme \eqref{appendix:scheme} also satisfies:
\begin{equation}
\label{appendix:scheme:entropy}
\vert K\vert\dfrac{d\eta(\xbar{\bbu}_K)}{dt}+\sum_{e}\vert e\vert \hbbg(\xbar{\bbu}_K,\xbar{\bbu}_{K_{e,-}}, \bbu_h)\leq 0
\end{equation}
Then we have the following result:
\begin{proposition}
\label{prop:entropy}
Under the condition of proposition \ref{prop:LxW:AF}, the limit solution satisfies the entropy inequality: for any $\varphi\in C^1_0(\R^d\times \R^+)$, $\varphi\geq 0$
$$\int_{\R^d\times \R^+}\bigg (\eta(\bbu)\dpar{\varphi}{t}+\bbg(\bbu)\cdot \nabla\varphi\bigg)\; d\bbx dt+\int_{\R^d}\eta(\bbu_0(\bbx))\varphi(\bbx,0)\; d\bbx\yolo{\geq} 0$$
\end{proposition}
As we can see, there is \emph{no} condition on the point values.
\begin{proof}
From \eqref{appendix:scheme:3}, we get:
$$
\dfrac{ d\eta(\bbu_\bsigma)}{ dt}+ \sum_{K,\bsigma\in K}\bbv(\bbu_\bsigma)^T\bomega_{K,\sigma} \bbPhi_{K,\bsigma}(\bbu_{ h})=0,\qquad \bbv(\bbu_\bsigma)=\nabla_\bbu{\eta}(\bbu_\bsigma)
$$
together with \eqref{appendix:scheme:entropy}. Let $\varphi\in C^1_0(\R^2\times \R^+)$, $\varphi\geq 0$.
We define
\begin{equation*}
\begin{split}\mathcal{I}&=\underbrace{\sum_K\vert K\vert\bigg (
\varphi(\bbx_K) \dfrac{ d\eta(\xbar\bbu_K)}{ dt}+ \sum_{j=1}^6\beta_j\big ( \varphi(\bsigma_j)-\varphi(\bbx_K)\big )\dfrac{ d\eta(\bbu_{ h})}{ dt}(\bsigma_j)\bigg)}_{\mathcal{I}_t}\\&\qquad +
 \underbrace{\yolo{\sum_K\bigg(}\varphi(\bbx_K) \oint_{\partial K}\hbbg_\bbn\;  d\ell+
\vert K\vert \sum_{j=1}^6\beta_j\big ( \varphi(\bsigma_j)-\varphi(\bbx_K)\big )\bigg (\sum_{K', \bsigma_j\in K'}\bbv(\bbu_{\bsigma_j})^T\bomega_{K',\bsigma_j}\bbPhi_{K',\bsigma_j}(\bbu_{ h})\bigg )\bigg )}_{\mathcal{I}_x}
\end{split}
\end{equation*}
We know that
\begin{equation*}
\begin{split}
\int_0^{+\infty}\mathcal{I}_t\; dt
&=\yolo{
- \sum_{K}\vert K\vert\bigg ( \varphi(\bbx_K,0)\eta(\xbar{\bbu}_K^0)+\sum_{j=1}^6\beta_j\big(\varphi(\bsigma_j,0)-\varphi(\bbx_K,0)\big)\eta(\bbu_0(\bsigma_j))\bigg)}\\
&\yolo{-
\int_0^{+\infty}\sum_{K}\vert K\vert \bigg (\partial_t\varphi(\bbx_K)\eta(\xbar{\bbu}_K)+\sum_{j=1}^6\beta_j\big(\partial_t\varphi(\bsigma_j)-\partial_t\varphi(\bbx_K)\big)\eta(\bbu_{\bsigma_j})\bigg )\; dt.}
\end{split}
\end{equation*}
This will converge to
\begin{equation*}
    \yolo{-\int_0^{+\infty}}\int_{\R^d}\yolo{\partial_t\varphi}\; \eta(\bbu)\; d\bbx dt-\int_{\R^d}\varphi(\bbx,0)\yolo{\eta(\bbu_0(\bbx))}\; d\bbx.
\end{equation*}
We get, using the same arguments as for the Lax-Wendroff proof,
\begin{equation*}
\begin{split}
\int_0^{+\infty }\mathcal{I}_x\; dt&= \int_0^{+\infty }\yolo{\sum_K}\varphi(\bbx_K) \oint_{\partial K}\hbbg_\bbn\;  d\ell\;dt\\
&\yolo{+\int_0^{+\infty} \vert K\vert \sum_{j=1}^6\beta_j\big ( \varphi(\bsigma_j)-\varphi(\bbx_K)\big )\bigg (\sum_{K', \bsigma_j\in K'}\bbv(\bbu_{\bsigma_j})^T\bomega_{K',\bsigma_j}\bbPhi_{K',\bsigma_j}(\bbu_{ h})\bigg )dt}
\end{split}
\end{equation*}
%because
%$$\dfrac{d\eta(\bbu_\bsigma)}{dt}=\bbv_\bsigma^T \dfrac{d\bbu_\bsigma}{dt}=\sum_{K, \bsigma\in K} \bbv_\bsigma^T \bomega_{K,\bsigma}\dfrac{d\bbu_\bsigma}{dt}$$
will converge to
$$\yolo{-}\int_0^{+\infty}\int_{\R^d}\nabla \varphi\cdot \bbg(\bbu)\; d\bbx dt.$$
In the end, we get the requested entropy inequality.
\end{proof}
%%%%%%%%%%%%%%%%%%%%%%%%%%%%%

\subsection{Bound preserving and entropy  properties of the first order schemes}
%%%%%%%%%%%%%%%%%%%%%%%%%%%%%

Concerning the bound preserving properties of these schemes, we first have the following lemma:
\begin{lemma}\label{lemma:pt}
We have:
\begin{enumerate}
\item $$\bbu_\bsigma^{n+1}=\sum\limits_{K, \bsigma\in K}\sum_{T_i^K, \bsigma\in T_i^K}\frac{\vert K\vert}{6\vert C_{\bsigma}\vert}\widetilde{\bbu}_\bsigma^{n+1}$$ with 
$$\widetilde{\bbu}_{\bsigma}^{n+1}=\bbu_\bsigma^n-\frac{6\Delta t}{\vert K\vert}\widehat{\bbPhi}_{\bsigma, T_i^K}^{LO}.$$ $\widehat{\bbPhi}_{\bsigma, T_i^K}^{LO}$ is defined by \eqref{Theta}
and
$$\vert C_{\bsigma}\vert\yolo{=\sum_{K,\bsigma\in K}\sum_{T_i^K,\bsigma\in T_i^K}\frac{\vert K\vert}{6}}=\sum_{K, \bsigma\in K}\frac{\vert K\vert }{3},$$
\item If the values of $\bbu$ attached to the vertices of $T_i^K$ (i.e. $2$ out of $\{\bbu_{\bsigma_1}, \ldots, \bbu_{\bsigma_6}\}$ and $\xbar{\bbu}_K$) are in $\DD$ at $t_n$, then $\widetilde{\bbu}_{\bsigma}^{n+1}\in \DD$ if $$\frac{\Delta t}{\vert K\vert}\leq \yolo{\frac{1}{6}}.$$
\item We have, under the same CFL condition, the entropy inequality:
$$\eta(\bbu_\bsigma^{n+1})\leq \eta(\bbu_\bsigma^n)-\Delta t \sum_{K, \bsigma\in K}\Psi_{\bsigma}^K$$ with
$$\Psi_\bsigma^K=\frac{1}{\vert C_\bsigma\vert}\sum_{T_i^K, \bsigma\in T_i^K} \widehat{\Psi}_{\bsigma, T_i^K}^{LO}$$
\begin{equation}
\begin{split}\widehat{\Psi}_{\bsigma, T_i^K}^{LO}&=\frac{1}{3}\sum_{\bsigma_l}\yolo{\frac{1}{2}}\bbg(\bbu_{\bsigma_l})\cdot \bbn_{\bsigma_l}^{T_j^K}+ \alpha_{T^K_i}\big (\eta(\bbu_\sigma)-\overline{\eta(\bbu)}_{T^K_i}\big )\\
  \overline{\eta(\bbu)}_{T^K_i}&=\frac{1}{3}\sum_{\bsigma_j\in T^K_i}\eta(\bbu_{\bsigma_l}).
\end{split}
\end{equation}
\end{enumerate}
\end{lemma}
\begin{proof}
We first have
\begin{equation*}
\begin{split}
\bbu_\bsigma^{n+1}&=\bbu_\bsigma^n-\Delta t \sum_{K, \bsigma\in K}\frac{1}{\vert C_\bsigma\vert}\bigg (\sum_{T_i^K, \bsigma\in T_i^K} \widehat{\bbPhi}_{\bsigma, \yolo{T_i^K}}^{LO}\bigg )\\
&=\sum_{K, \bsigma\in K} \frac{\vert K\vert}{3\vert C_{\bsigma}\vert}\bigg ( \bbu_\bsigma^n-\frac{3\Delta t}{\vert K\vert}\sum_{T_i^K, \bsigma\in T_i^K} \widehat{\bbPhi}_{\bsigma, \yolo{T_i^K}}^{LO}\bigg )\\
&=\sum_{K, \bsigma\in K}\yolo{\sum_{T_i^K, \bsigma\in T_i^K}} \frac{\vert K\vert}{6\vert C_{\bsigma}\vert}\bigg (
\bbu_\bsigma^n-\frac{6\Delta t}{\vert K\vert}\widehat{\bbPhi}_{\bsigma, \yolo{T_i^K}}^{LO}\bigg )
\end{split}
\end{equation*}
because $\vert T_i^K\vert = \vert K\vert/6$. This show the first item.

Then as in \cite{GuermondPopov,GuermondNazarovPopov} (see also \cite{KuzminMollerShadidShaskhov2010,LohmannKuzminShadidMabuza2017,CotterKuzmin2016} for earlier and very similar contributions),
\yolo{
\begin{equation*}
\begin{split}
&\bbu_\bsigma^n-\frac{6\Delta t}{\vert K\vert}\widehat{\bbPhi}_{\bsigma, \yolo{T_i^K}}^{LO}=
\bbu_\bsigma^n-\frac{\Delta t}{\vert K\vert} \sum_{\bsigma_l\in T_i^K} \bigg( \big (\bbf(\bbu_{\bsigma_l}^n)-\bbf(\bbu_{\bsigma}^n)\big )\cdot\bbn_{\bsigma_l}^{T_{i,K}}+2\alpha(\bbu_\bsigma^n-\bbu_{\bsigma_l}^n)\bigg )\\
&=\big ( 1-\frac{6\Delta t}{\vert K\vert}\alpha)\bbu_\bsigma^n+\frac{2\Delta t\alpha}{\vert K\vert}\sum_{\bsigma_l\in T_i^K}\bigg ( 
\frac{\bbu_\bsigma^n+\bbu_{\bsigma_l}^n}{2}- \big ( \bbf(\bbu_{\bsigma_l}^n)-\bbf(\bbu_\bsigma^n)\big )\cdot\frac{\bbn_{\bsigma_l}^{T_{i,K}}}{2\alpha}\bigg )+\frac{\Delta t\alpha}{\vert K\vert}\sum_{\bsigma_l\in T_i^K,\bsigma_l\neq\bsigma}\bbu_{\bsigma_l}^n
\end{split}
\end{equation*}}
This shows that under that if $\alpha$ is larger than  $\alpha(\bbu_\bsigma, \bbu_{\bsigma'}, \frac{\bbn_{\bsigma'^{T_{i,K}}}}{\Vert\bbn_{\bsigma'^{T_{i,K}}}\Vert})\Vert\bbn_{\bsigma'^{T_{i,K}}}\Vert $ and if $\Delta t$ satisfies
$$\frac{\yolo{6}\Delta t}{\vert K\vert}\alpha\leq 1$$ 
then under our assumptions on $\bbu$, 
$$\frac{\bbu_\bsigma+\bbu_{\bsigma_l}}{2}- \big ( \bbf(\bbu_{\bsigma_l})-\bbf(\bbu_\bsigma)\big )\cdot\frac{\bbn_{\bsigma_l}^{T_{i,K}}}{2\alpha}\in \DD.$$

Following again \cite{GuermondPopov,GuermondNazarovPopov} and under the same CFL condition, we see that
$$\eta(\widetilde{\bbu}_\bsigma^{n+1})\leq \eta({\bbu}_\bsigma^{n})-6\frac{\Delta t}{\vert K\vert}
\bigg (\frac{1}{3}\sum_{\bsigma_l\in T_i^K} \frac{1}{2}\bbg(\bbu_{\bsigma_l})\cdot \bbn_{\bsigma_l}^{T_{i,K}}+\alpha \big (\eta(\bbu_\bsigma)-\overline{\eta(\bbu)}_{T_i^K}\big )$$
and the rest follows.
\end{proof}

We have a similar lemma for the average which proof is similar and very classical:
\begin{lemma}\label{lemma:ave}
The finite volume scheme:
$$\bbu_K^{n+1}=\bbu_K^{n+1}-\frac{\Delta t}{\vert K\vert}\sum_e \vert e\vert \hbbf_{\bbn_e}(\bbu_K,\bbu_{K'})$$ where $\hbbf_{\bbn_e}$ is the Rusanov flux is invariant domain preserving and  entropy satisfying with
$$\eta(\bbu_K^{n+1})\leq \eta(\bbu_K^{n})-\frac{\Delta t}{\vert K\vert}\sum_e\vert e\vert \hbbg_{\bbn_e}(\bbu_K^{n},\bbu_{K'}^{n})$$
where
$$\hbbg_{\bbn_e}(\bbu_K,\bbu_{K'})=\frac{1}{2}\bigg ( \bbg(\bbu_K)+\bbg(\bbu_{K'})\bigg )\cdot \bbn_e+\alpha_e\big ( \eta(u_K^n)-\xbar{\eta(u)})$$
and 
$$\xbar{\eta(u)}=\frac{1}{2}\bigg ( \eta(\bbu_K)+\eta(\bbu_{K'})\bigg )$$
under the CFL condition 
$$\Delta \frac{\max_e\vert e\vert \alpha_e}{\vert K\vert}\leq \frac{1}{2}.$$
\end{lemma}
\begin{proof}See \cite{GuermondNazarovPopov}.
\end{proof}
%%%%%%%%%%%%%%%%%%%%%%%%%%%%%%%%%%%%%%%%%%
%\input{entropy_relations}

\subsection{A semi-discrete entropy satisfying scheme}
Here we look for a scheme that will satisfy a semi-discrete entropy inequality. From the truncation error and the formal convergence analysis, we see that it is enough to have an entropy inequality for the average only and we will write:
\begin{equation}
    \label{ent:eq:1}\vert K\vert \dfrac{d\xbar{\bbu}_K}{dt}+\sum_{e\subset\text{edge of }K} \vert e\vert \hbbf_{\bbn_e}(\xbar\bbu_{K}, \xbar\bbu_{K_{e,-}}, \bbu_{|e})=0\end{equation}
with, on the edge $e$ of $K$, we assume
$$\vert e\vert\hbbf_{\bbn_e}(\xbar{\bbu}_{K}, \xbar{\bbu}_{K_{e,-}}, \xbar\bbu_{\vert e})=(1-\ell_e) \int_e\bbf(\bbu)\cdot\bbn_e\;d\gamma+\ell_e\vert e\vert\hbbf_\bbn(\xbar\bbu_{K}, \xbar\bbu_{K_{e,-}}):=\vert e\vert \big ((1-\ell_e)\hbbf^{HO}_{\bbn_e}+\ell_e\hbbf^{LO}_{\bbn_e}\big )$$ with obvious notations.
As usual, $K_{e,-}$ is the element on the other side of $e$. The parameter $\ell_e$ will be chosen to satisfy a semi-discrete entropy inequality. Since there is no ambiguity, we drop the subscript "$e$".

The guideline is Tadmor's work. Setting $\bbv=\nabla_\bbu\eta$ the entropy variable, we write
$$g=\bbv^T\bbf-\psi$$ where $\psi$ is the potential. 
From \eqref{ent:eq:1}, we get
$$\vert K\vert \dfrac{d\eta(\xbar{\bbu}_K)}{dt}+\sum_{e}\vert e\vert \langle \bbv_K, \hbbf_{\bbn_e}\rangle=0.$$
Then we write
\begin{equation*}
    \begin{split}
\sum_{e}\vert e\vert \langle \bbv_K, \hbbf_{\bbn_e}\rangle&=
\sum_e \vert e\vert \bigg ( \big \langle\frac{\bbv_K+\bbv_{K_{e,-}}}{2},\hbbf_{\bbn_e}\rangle +\big \langle \frac{\bbv_K-\bbv_{K_{e,-}}}{2},\hbbf_{\bbn_e}\rangle\bigg )\\
&=\sum_e \vert e\vert \bigg ( \big \langle\frac{\bbv_K+\bbv_{K_{e,-}}}{2},\hbbf_{\bbn_e}\rangle -\frac{\psi(\xbar{\bbu}_K)-\psi(\xbar{\bbu}_{K_{e,-}})}{2}\bbn_e+
\frac{\psi(\xbar{\bbu}_K)+\psi(\xbar{\bbu}_{K_{e,-}})}{2}\bbn_e
+\big \langle \frac{\bbv_K\yolo{-}\bbv_{K_{e,-}}}{2},\hbbf_{\bbn_e}\rangle\bigg )\\
&=\sum_e \vert e\vert \bigg ( \big \langle\frac{\bbv_K+\bbv_{K_{e,-}}}{2},\hbbf_{\bbn_e}\rangle -\frac{\psi(\xbar{\bbu}_K)+\psi(\xbar{\bbu}_{K_{e,-}})}{2}\bbn_e\bigg )\\
&\qquad+ \sum_e \vert e\vert \bigg (
\big \langle \frac{\bbv_K-\bbv_{K_{e,-}}}{2},\hbbf_{\bbn_e}\rangle
-\frac{\psi(\xbar{\bbu}_K)-\psi(\xbar{\bbu}_{K_{e,-}})}{2}\bbn_e\bigg )\\
&=\sum_e \vert e\vert \hbbg_{\bbn_e}+\sum_e \vert e\vert \bigg (
\big \langle \frac{\bbv_K-\bbv_{K_{e,-}}}{2},\hbbf_{\bbn_e}\rangle
-\frac{\psi(\xbar{\bbu}_K)-\psi(\xbar{\bbu}_{K_{e,-}})}{2}\bbn_e\bigg )
\end{split}
\end{equation*}
where the entropy flux is
$$\hbbg_{\bbn_e}=\big \langle\frac{\bbv_K+\bbv_{K_{e,-}}}{2},\hbbf_{\bbn_e}\rangle -\frac{\psi(\xbar{\bbu}_K)+\psi(\xbar{\bbu}_{K_{e,-}})}{2}\bbn_e.$$
Hence a enough condition is
\begin{equation}
    \label{entropy:cond}
    \big \langle \frac{\bbv_K-\bbv_{K_{e,-}}}{2},\hbbf_{\bbn_e}\rangle
-\frac{\psi(\xbar{\bbu}_K)-\psi(\xbar{\bbu}_{K_{e,-}})}{2}\bbn_e\geq 0
\end{equation}
on each edge $e$.
This condition writes
\begin{equation}
    \label{entropy:condition2}
   \Theta(\ell_e):= \big \langle \bbv_K-\bbv_{K_{e,-}}, \hbbf^{LO}_{\bbn_e}\big \rangle-\big ( \psi(\xbar{\bbu}_K)-\psi(\xbar{\bbu}_{K_{e,-}})\big )\bbn_e+\ell_e\big \langle \bbv_K-\bbv_{K_{e,-}}, \hbbf^{HO}_{\bbn_e}-\hbbf^{LO}_{\bbn_e}\big \rangle\geq 0
\end{equation}
By construction, the Rusanov flux
$$\hbbf_{\bbn_e}^{LO}=\frac{1}{2}\big (\bbf(\xbar{\bbu}_K)\bbn_e+
\bbf(\xbar{\bbu}_{K_{e,-}})\bbn_e+\alpha_e (\xbar{\bbu}_K-\xbar{\bbu}_{K_{e,-}})\big ),$$
is known that \yolo{entropy stable} if \yolo{$\alpha_e$ is the largest speed estimated using the states $\xbar\bbu_K$ and $\xbar{\bbu}_{K,e}$ and thus $\Theta(0)\geq 0$. Since $\Theta$ depends affinely on $\ell_e$,} there exists a largest $\ell^\star$ such that for $\ell_e\in [0,\ell^\star]$, $\Theta(\ell_e)\geq 0$. 

To simplify, we set
$$A=\big \langle \bbv_K-\bbv_{K_{e,-}}, \hbbf^{LO}_{\bbn_e}\rangle-\big ( \psi(\xbar{\bbu}_K)-\psi(\xbar{\bbu}_{K_{e,-}}\big )\bbn_e, \quad B=-\big \langle \bbv_K-\bbv_{K_{e,-}}, \hbbf^{HO}_{\bbn_e}-\hbbf^{LO}_{\bbn_e}\big \rangle$$
with $A\geq 0$. The problem is to study $A-\ell B\geq 0$.
We have $A\geq \ell B$ and want $\ell\in [0,1]$. 
\begin{figure}[htbp]
    \centering
    \begin{tikzpicture}
    \begin{axis}[
        % --- AJUSTEMENT DE LA TAILLE ICI ---
        width=0.49\textwidth,       % Ajustez la largeur (ex: 8cm, 0.5\textwidth)
 %       height=6.5cm,    % Ajustez la hauteur proportionnellement
        % ------------------------------------
        axis lines = middle,
        xlabel = {$x$},
        ylabel = {$y$},
        xmin = -2, xmax = 4.5,
        ymin = -0.5, ymax = 2.5,
        grid = major,
        restrict y to domain=0:5,
        legend pos = north east,
        yticklabel style={yshift=0.45cm, anchor=north east},
        xticklabel style={xshift=-0.1cm, anchor=north east},
        legend image code/.code={
            \ifnum\plotnum=0
                \draw[color=black, thick] (0cm,0cm) -- (0.6cm,0cm);
 %           \fi
 %           \ifnum\plotnum=1
  %              \draw[color=blue, dashed, ultra thick] %(0cm,0cm) -- (0.6cm,0cm);
            \fi
        }
    ]
        % --- Tracé de g(x) (Noir plein) ---
        \addplot [domain=-2:0, samples=2, color=black, thick] {1};
        \addplot [domain=0:4, samples=2, color=black, thick, forget plot] {1 - x/4};
        \addplot [domain=4:4.5, samples=2, color=black, thick, forget plot] {0};
        
        % --- Tracé de f(x) (Pointillés bleus) ---
        \addplot [domain=-4:0, samples=2, color=blue, dashed, ultra thick] {1};
        \addplot [domain=0.22:4.5, samples=100, color=blue, dashed, ultra thick, forget plot] {1/x};
        
        \legend{$\varphi_{1/4}$}
   %     \draw[blue, fill=blue] (axis cs:0,1) circle (2.5pt);
    \end{axis}
    \end{tikzpicture}
    %figure 2
        \begin{tikzpicture}
    \begin{axis}[
        % --- AJUSTEMENT DE LA TAILLE ICI ---
        width=0.49\textwidth,       % Ajustez la largeur (ex: 8cm, 0.5\textwidth)
 %       height=6.5cm,    % Ajustez la hauteur proportionnellement
        % ------------------------------------
        axis lines = middle,
        xlabel = {$x$},
        ylabel = {$y$},
        xmin = -2, xmax = 2.5,
        ymin = -0.5, ymax = 2.5,
        grid = major,
        restrict y to domain=0:5,
        legend pos = north east,
        yticklabel style={yshift=0.25cm, anchor=north east},
        xticklabel style={xshift=-0.15cm, anchor=north east},
        legend image code/.code={
            \ifnum\plotnum=0
                \draw[color=black, thick] (0cm,0cm) -- (0.6cm,0cm);
            \fi
%            \ifnum\plotnum=1
%                \draw[color=blue, dashed, ultra thick] (0cm,0cm) -- (0.6cm,0cm);
 %           \fi
        }
    ]
        % --- Tracé de g(x) (Noir plein) ---
        \addplot [domain=-2:0, samples=2, color=black, thick] {1};
        \addplot [domain=0:2, samples=2, color=black, thick, forget plot] {1 - x/2};
        \addplot [domain=2:2.5, samples=2, color=black, thick, forget plot] {0};
        
        % --- Tracé de f(x) (Pointillés bleus) ---
        \addplot [domain=-4:0, samples=2, color=blue, dashed, ultra thick] {1};
        \addplot [domain=0.22:4.5, samples=100, color=blue, dashed, ultra thick, forget plot] {1/x};
        
        \legend{$\varphi_{1/2}$}
%        \draw[blue, fill=blue] (axis cs:0,1) circle (2.5pt);
    \end{axis}
    \end{tikzpicture}
    \caption{\label{ell}Graphical representation of the condition $A-\ell B\geq 0$ and the functionals used in this paper to estimate $\ell_e$.}
    \label{fig:mes_fonctions}
\end{figure}
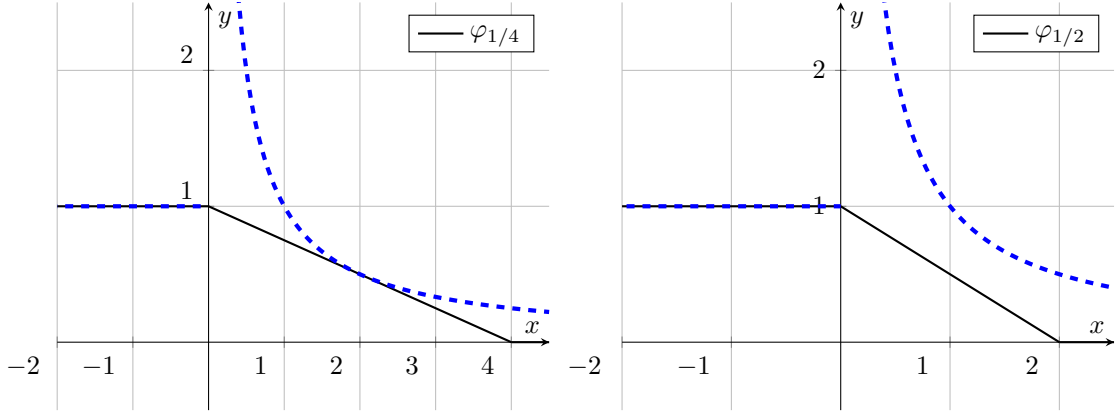
This condition express that $(B,\ell)$ should be below the dotted curves on Figure \ref{ell}. A first solution could be 
$$\ell=\left\{ \begin{array}{ll}
1 & \text{ if } B\leq A\\
\frac{A}{B} & \text{ else.}
\end{array}
\right .$$
In practice we have found it is better to stay strictly below the dotted hyperbola, and use 
\begin{equation}\label{entropy:min}
    \varphi_\alpha=\left \{ \begin{array}{cc}
    1 & B\leq 0 \\
    1-\alpha B & \text{else,}
\end{array}\right .
\end{equation}
with $\alpha\geq \frac{1}{4}$. In practice we have taken $\alpha=\frac{1}{2}$, see Figure \ref{ell}.

\section{Numerical examples}
In order to evaluate this entropy blending parameter, we have chosen the KPP test case \cite{KPP} because it is extremely sensitive on how well the entropy condition is satisfied.
The PDE is
\begin{equation}
    \dpar{u}{t}+\text{ div }\bbf(u)=0
    \end{equation} with
    $$\bbf(u)=\big ( \cos u, \sin u\big ).$$
    The initial condition is 
\begin{equation*}
    u(\bbx,0)=\left\{
    \begin{aligned}
    &\frac{7}{2}\pi \approx 10.995&&\mbox{if}~\Vert\bbx-(0,0.5)\Vert\leq 1,\\
    &\frac{\pi}{4}\approx 0.785 &&\mbox{else}.
    \end{aligned}\right.
\end{equation*} The problem is set in $[-2,2]^2$.
The unique entropy solution at the final time $t=1$ exhibits a two-dimensional rotational wave structure.  A key difficulty of this test is ensuring that the numerical scheme introduces sufficient dissipation so that the approximations converge to the correct entropy solution, and is not too dissipative so that one gets a sharp representation of the discontinuities.

In the numerical simulations, we proceed as follows for the average values.
Four different methods have been tested, they differ by how $\ell$ is evaluated:
\begin{itemize}
\item If the BP condition is only applied as in \cite{BP_Pampa_VEM},
\item if the BP condition of \cite{BP_Pampa_VEM} and the OE condition of \cite{OEPampa} are used,
\item if the BP and OE conditions are use in addition to the semi-discrete entropy condition with $\varphi_{1/2}$,
\item if only the semi-discrete condition is used.
\end{itemize}
 For the point values, except for the first order results, we keep the blending parameter of \cite{BP_Pampa_VEM} or \cite{OEPampa}.

\begin{figure}[htbp]
\begin{center}
\subfigure[point value]{\includegraphics[width=0.49\textwidth,clip=]{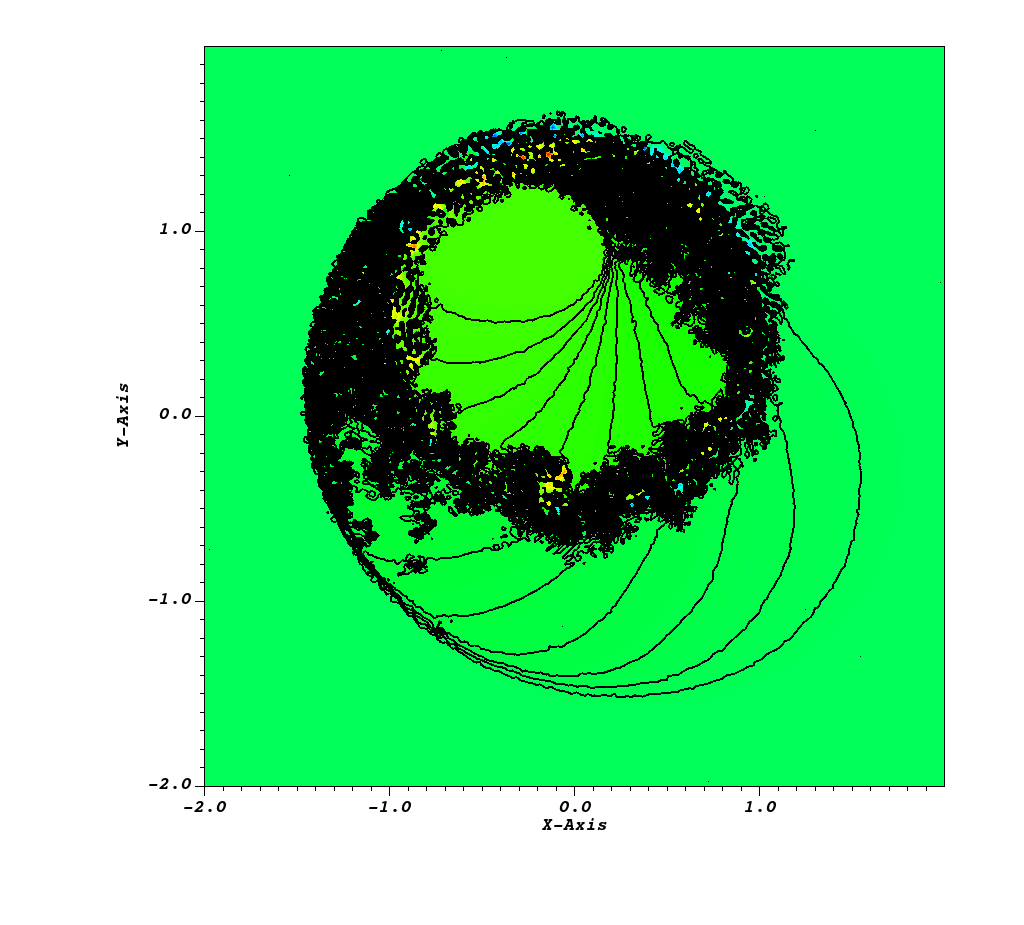}}
\subfigure[average value]{\includegraphics[width=0.49\textwidth,clip=]{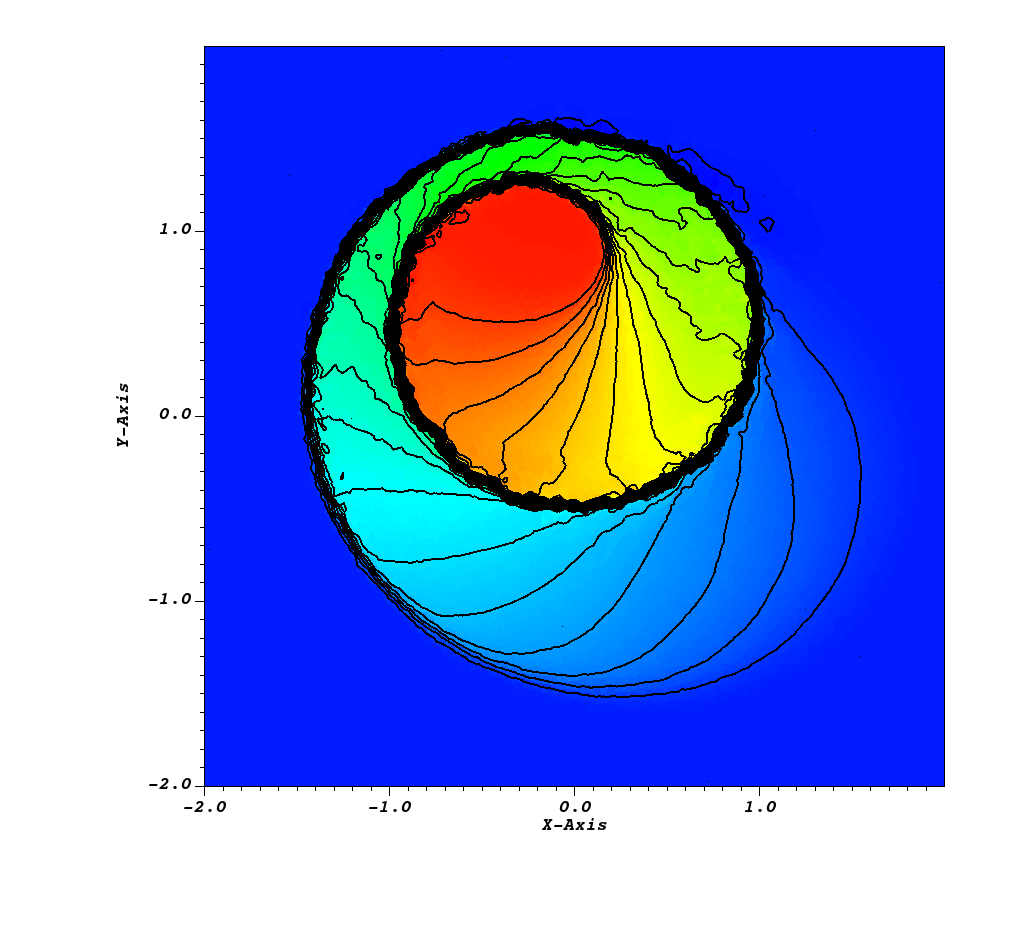}}
\caption{\label{fig:KPP:tadmor} KPP solution when only the 'entropy' limiter is used on the average, and the fully high order scheme is used for the point value update. We take 30 isolines in both cases, $\xbar{\bbu}\in [0.477, 11.338]$ and $\bbu_\bsigma\in [-25.562,      39.288]$.}
\end{center}
\end{figure}
 The first results plotted in Figure \ref{fig:KPP:tadmor} shows the solution when the entropic blending is applied on the average update only. The full high order scheme is used for the point update. Of course the solution is very wiggly, especially the point value one, but clearly the entropy solution is captured. This result show that the entropy blending described in this paper is effective: one only need to control the average. One can also compare with Figure \ref{fig:KPP:pt}-(a) and \ref{fig:KPP:ave}-(a) when only the BP blending of \cite{BP_Pampa_VEM} is used\footnote{This seems to be in contradiction with a statement of \cite{BP_Pampa_VEM}, this is not. We have realized after the publication of this paper that we had chosen the wrong global extrema. We had chosen a value of the maximum strictly less that the true value, hence the scheme was first order initially on the discontinuity. This was enough to get the correct solution.}.

In what follow, unless explicitly stated, the BP blending of \cite{BP_Pampa_VEM} or \cite{OEPampa} are now used for the point values update in order to better control the point values extrema. 
The results are displayed on Figures \ref{fig:KPP:pt}, \ref{fig:KPP:ave} and \ref{fig:KPP:zoom}.
We first notice that applying the BP limiter only will not lead to the correct entropy condition. If we only use the entropy limiter (subfigure d), we may get under- or overshoot, but we get the entropy solution. This shows that the method is effective. \yoloo{Compared with the results in Figure \ref{fig:KPP:tadmor}, where the fully high order scheme is used for the point value update. Here, the point value update is additionally blended with the entropy limiter.} \yoloo{Therefore, as expected, the solution of the point values is much cleaner.}
\begin{figure}[htbp]
\begin{center}
\subfigure[BP]{\includegraphics[width=0.49\textwidth,clip=]{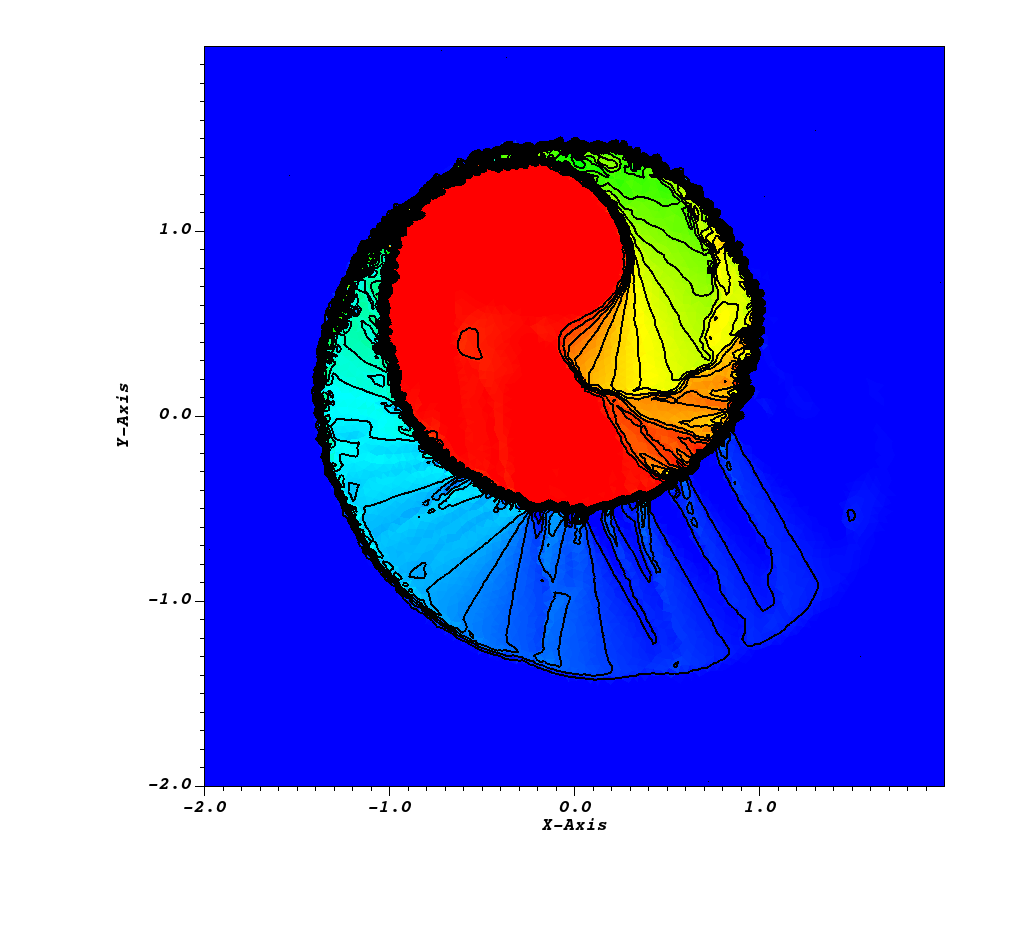}}
\subfigure[BP OE]{\includegraphics[width=0.49\textwidth,clip=]{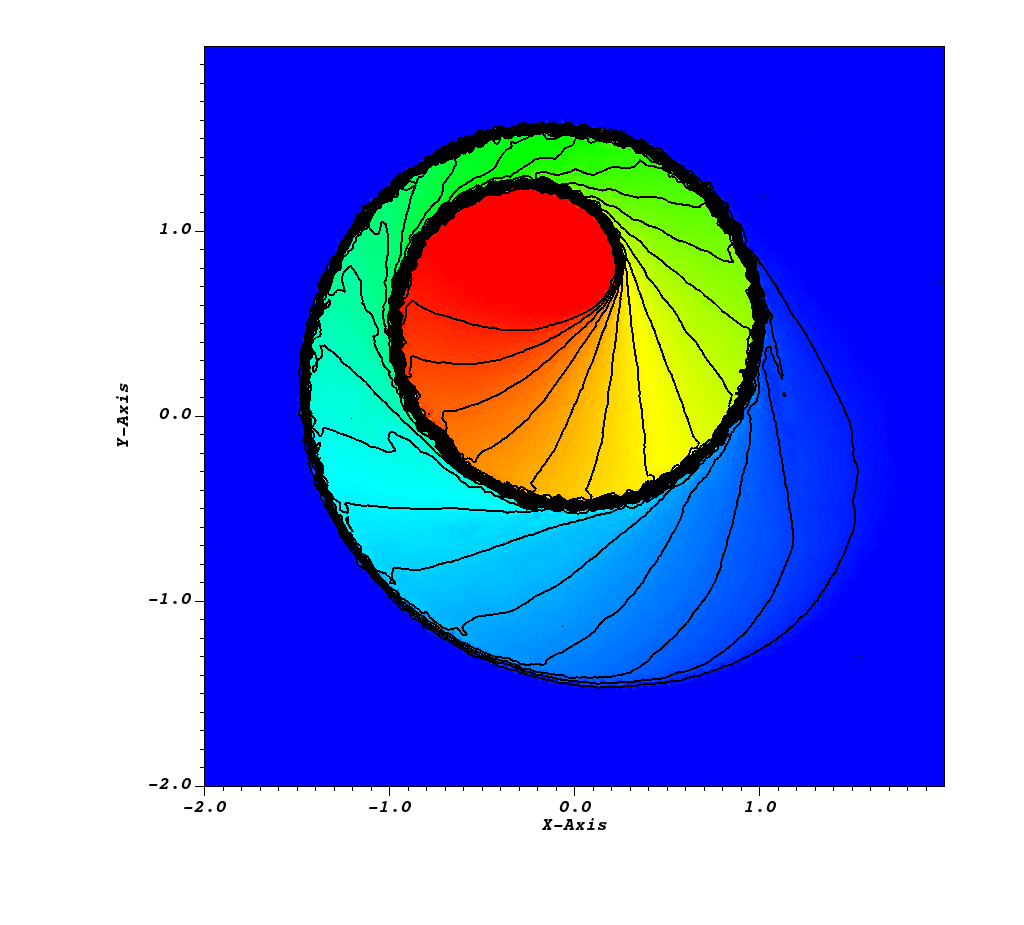}}
\subfigure[BP OE Tadmor]{\includegraphics[width=0.49\textwidth,clip=]{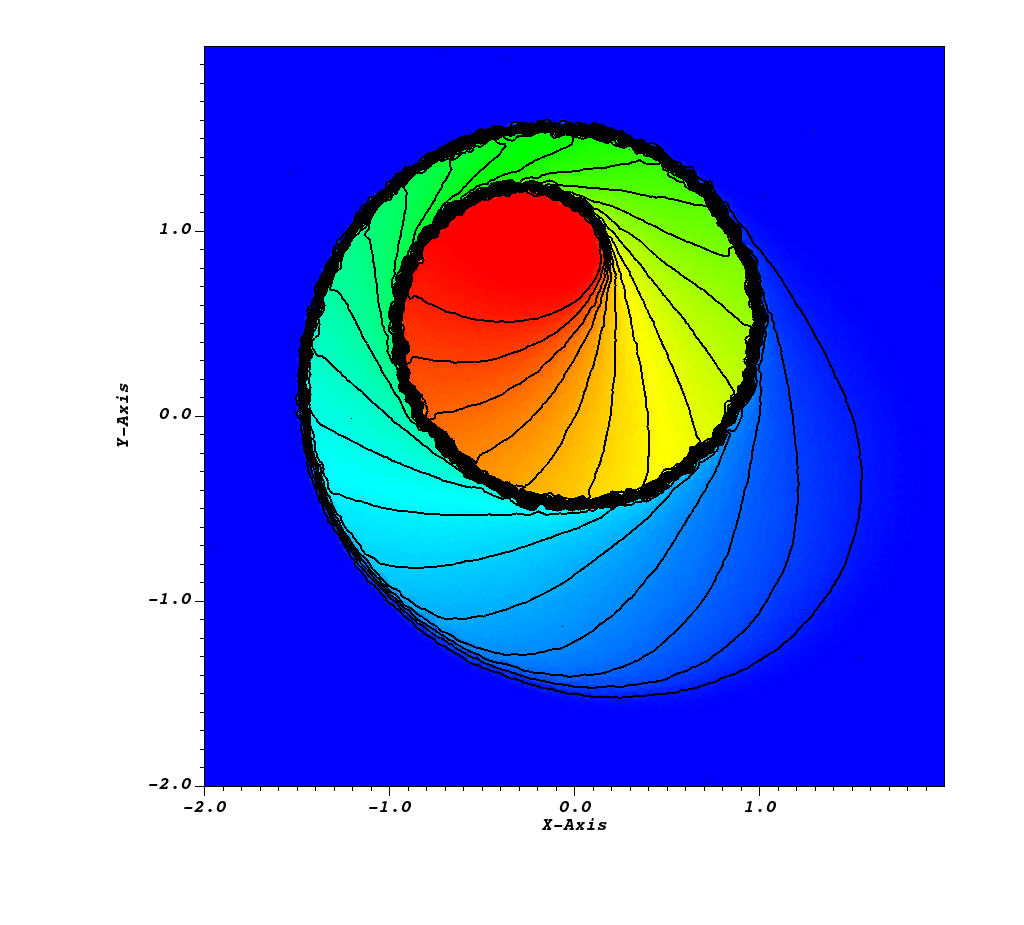}}
\subfigure[Tadmor]{\includegraphics[width=0.49\textwidth,clip=]{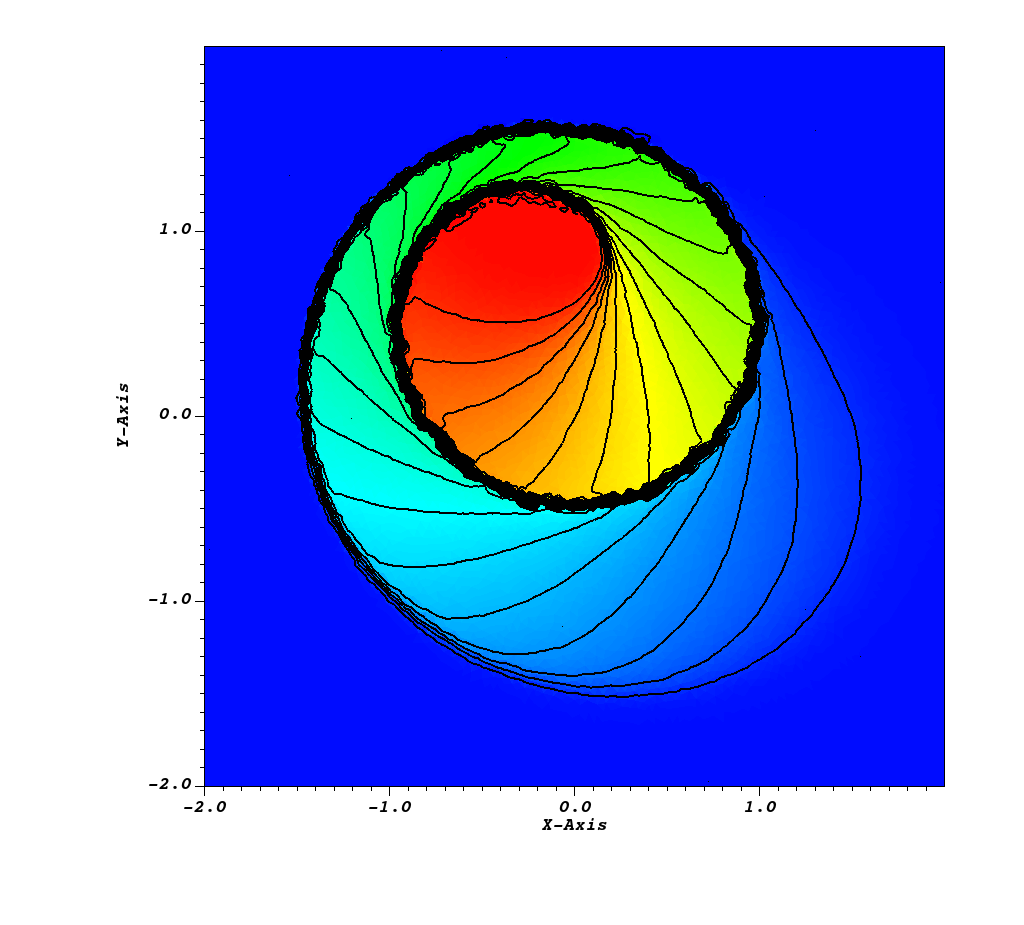}}
\end{center}
\caption{\label{fig:KPP:pt}KPP problem, average values, (a): BP solution, (b) BP with OE, (c): BP with OE and entropy limiter \eqref{entropy:min}. 30 iso-lines between $0.7$ and $11$.}
\end{figure}
The other variant are almost undistinguishable. From Figure \ref{fig:KPP:ave}, it seems however that the method that combines all the criteria provides the smoothest results, but BP+OE is already excellent.
\begin{figure}[H]
\begin{center}
\subfigure[BP]{\includegraphics[width=0.49\textwidth]{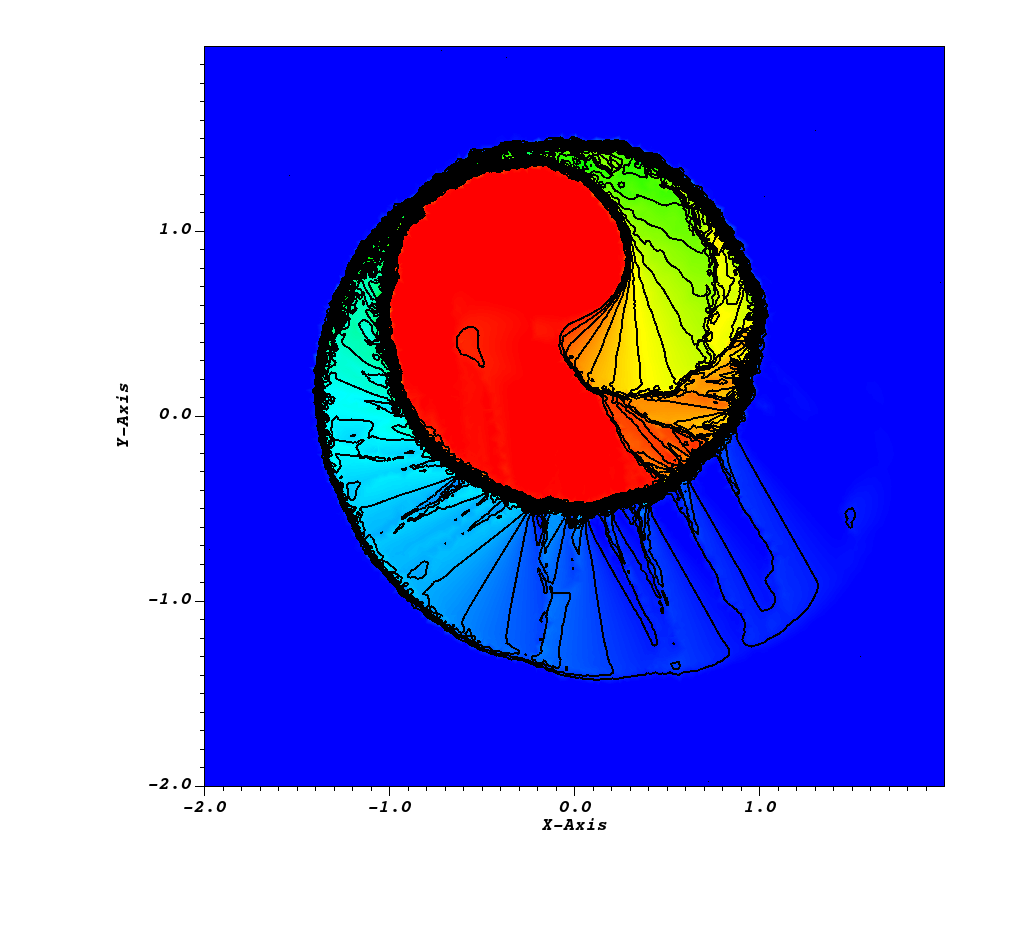}}
\subfigure[BP OE]{\includegraphics[width=0.49\textwidth]{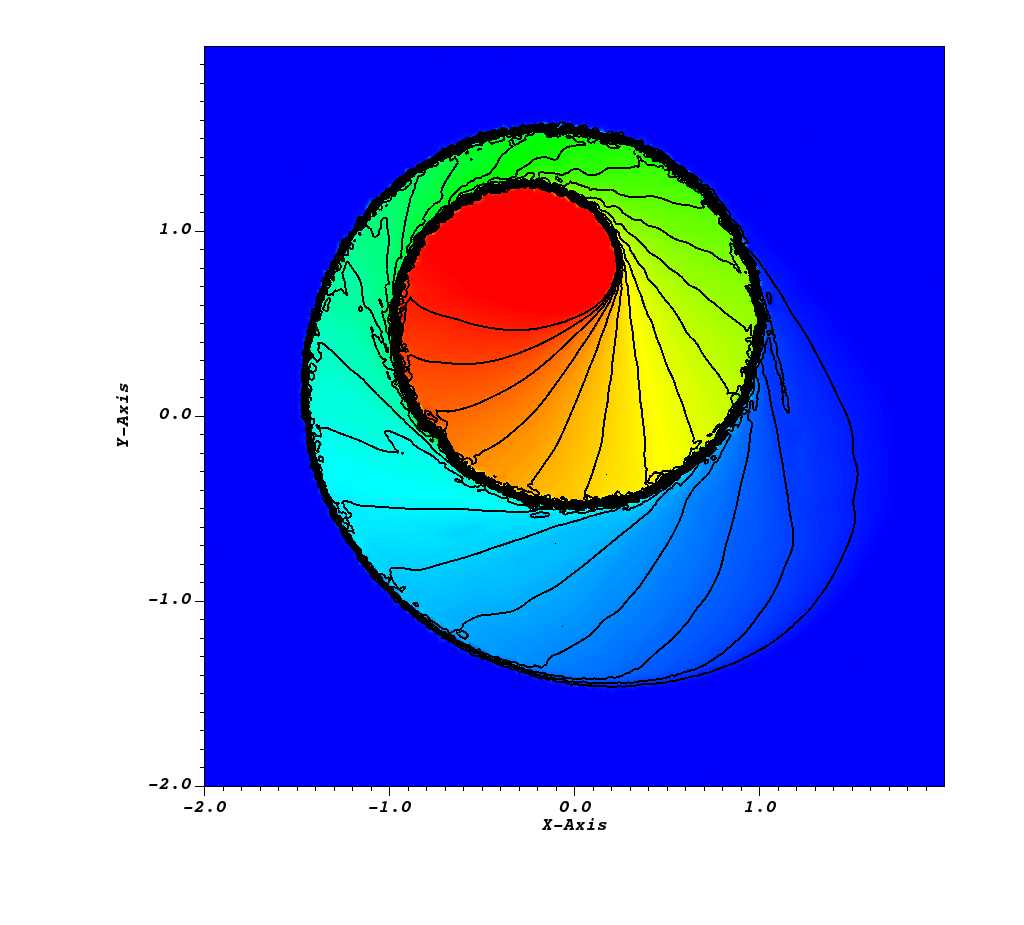}}
\subfigure[BP OE Tadmor]{\includegraphics[width=0.49\textwidth]{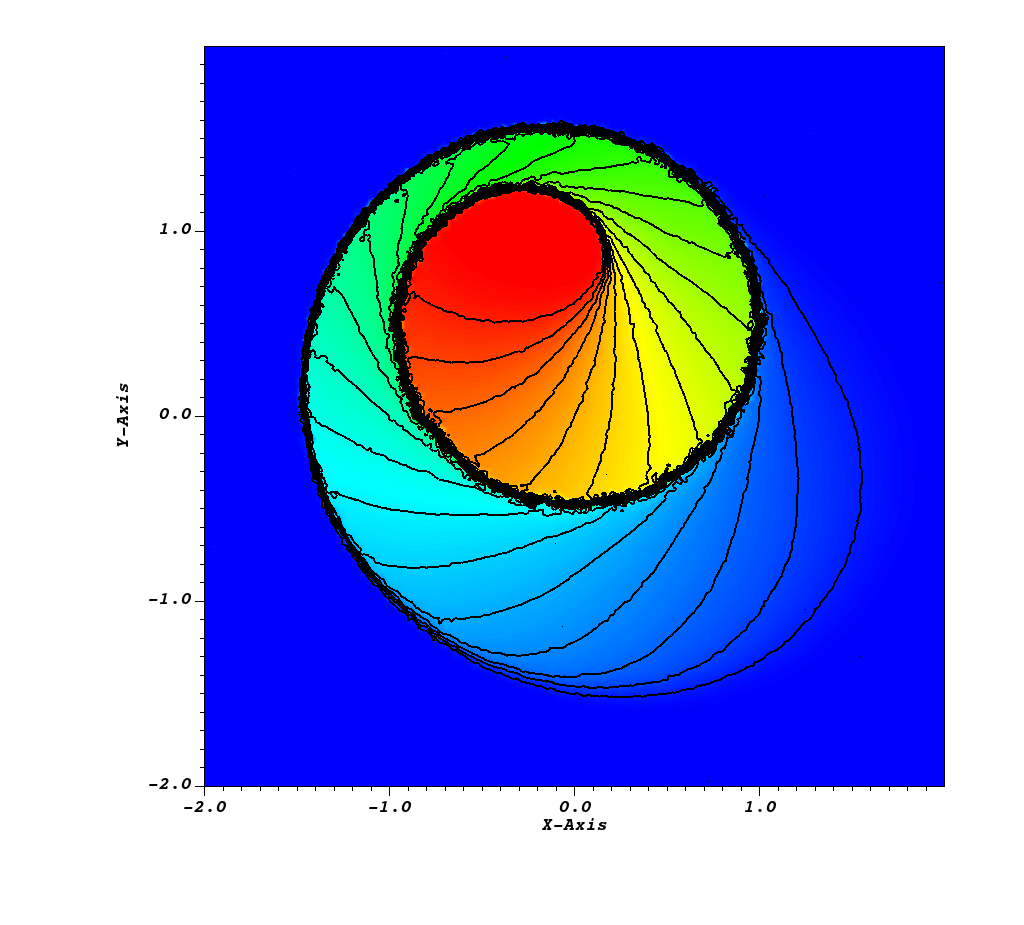}}
\subfigure[Tadmor]{\includegraphics[width=0.49\textwidth]{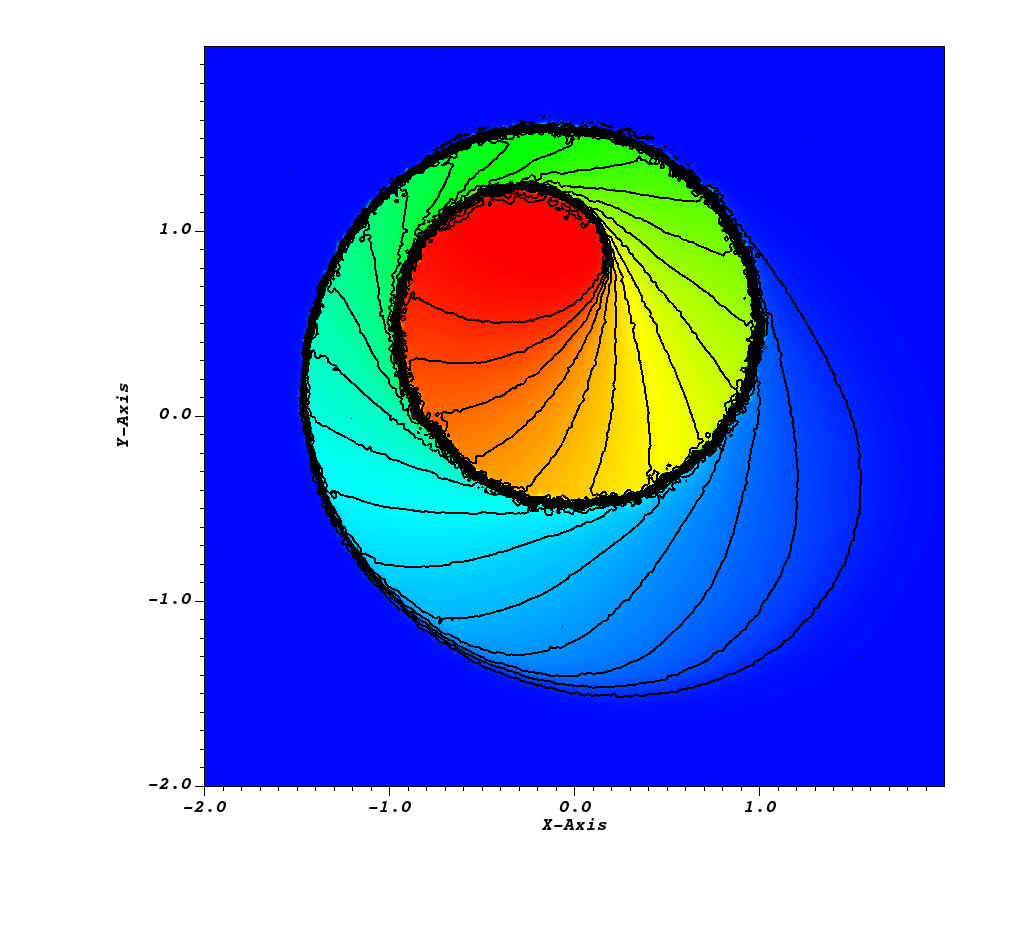}}
\end{center}
\caption{\label{fig:KPP:ave}KPP problem, point values, (a): BP solution, (b) BP with OE, (c): BP with OE and entropy limiter \eqref{entropy:min}. 30 isolines between $0.7$ and $11$.}
\end{figure}
The mesh used for solving this KPP problem is displayed in Figure \ref{fig:KPP:zoom} together with the solution with the entropy correction only, the BP-OE-Tadmor parameter and the first order solution. This also show the gain in accuracy across the discontinuities.
\begin{figure}[H]
\begin{center}
   \subfigure[point,Tadmor]{\includegraphics[width=0.45\textwidth]{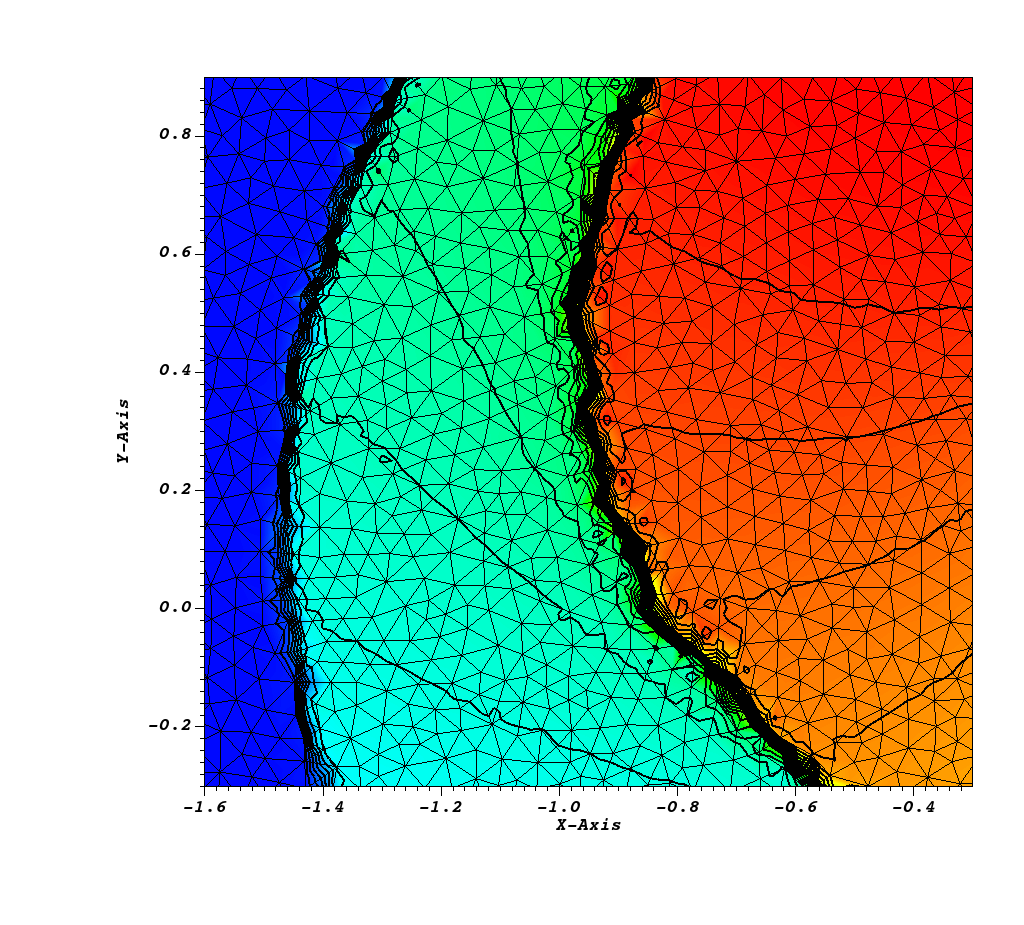}} 
   \subfigure[ave,Tadmor]{\includegraphics[width=0.45\textwidth]{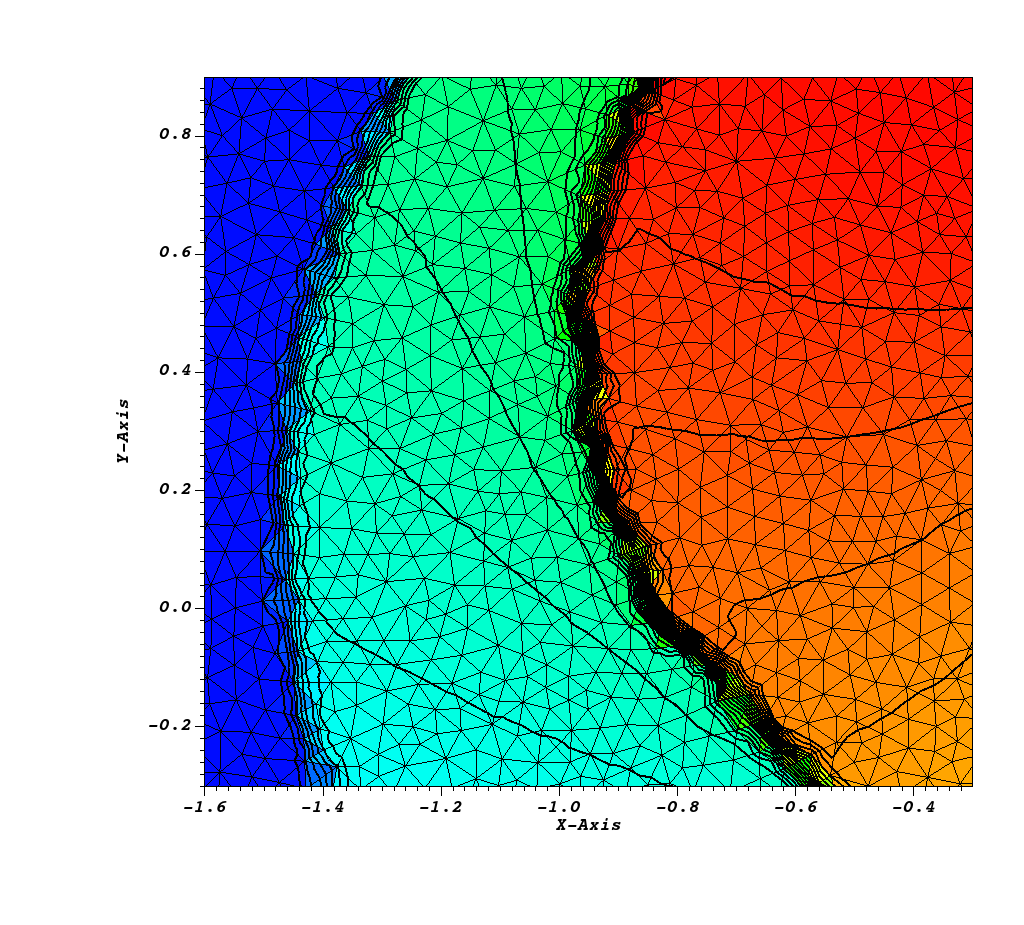}}
    \subfigure[point, BP OE Tadmor]{\includegraphics[width=0.45\textwidth]{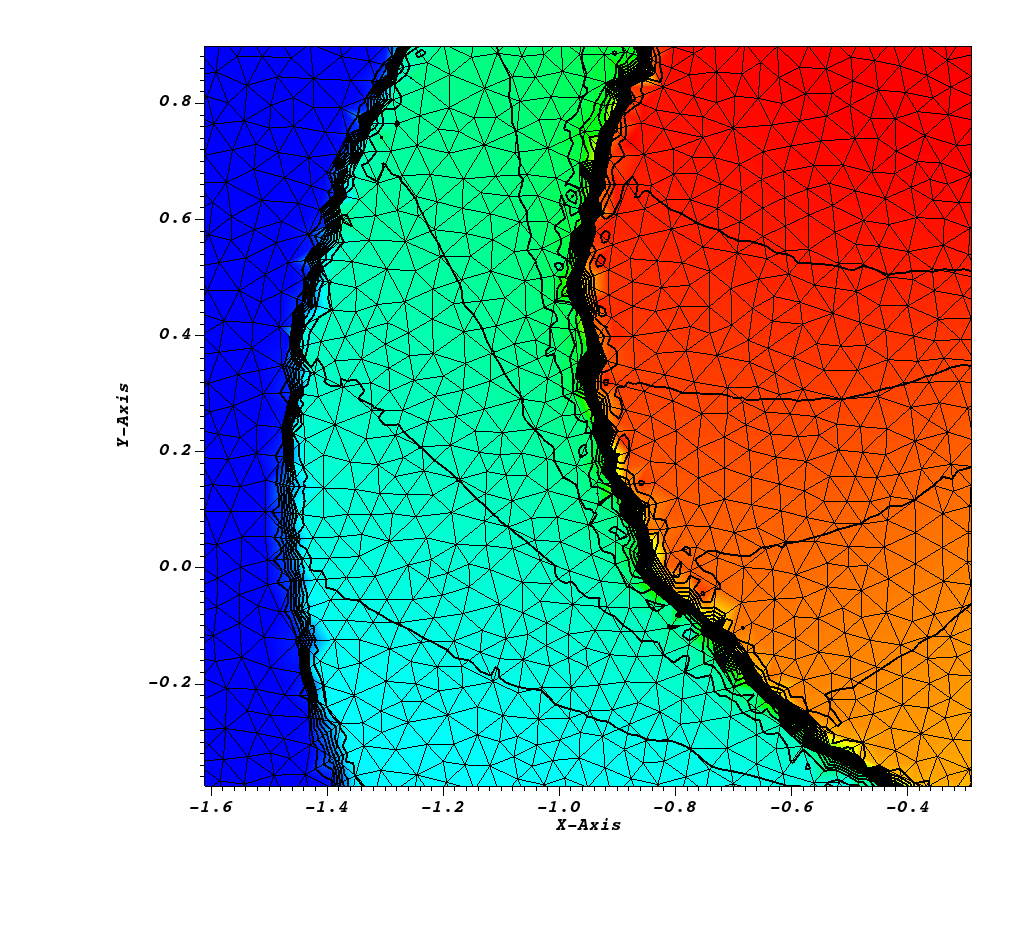}} 
    \subfigure[ave, BP OE Tadmor]{\includegraphics[width=0.45\textwidth]{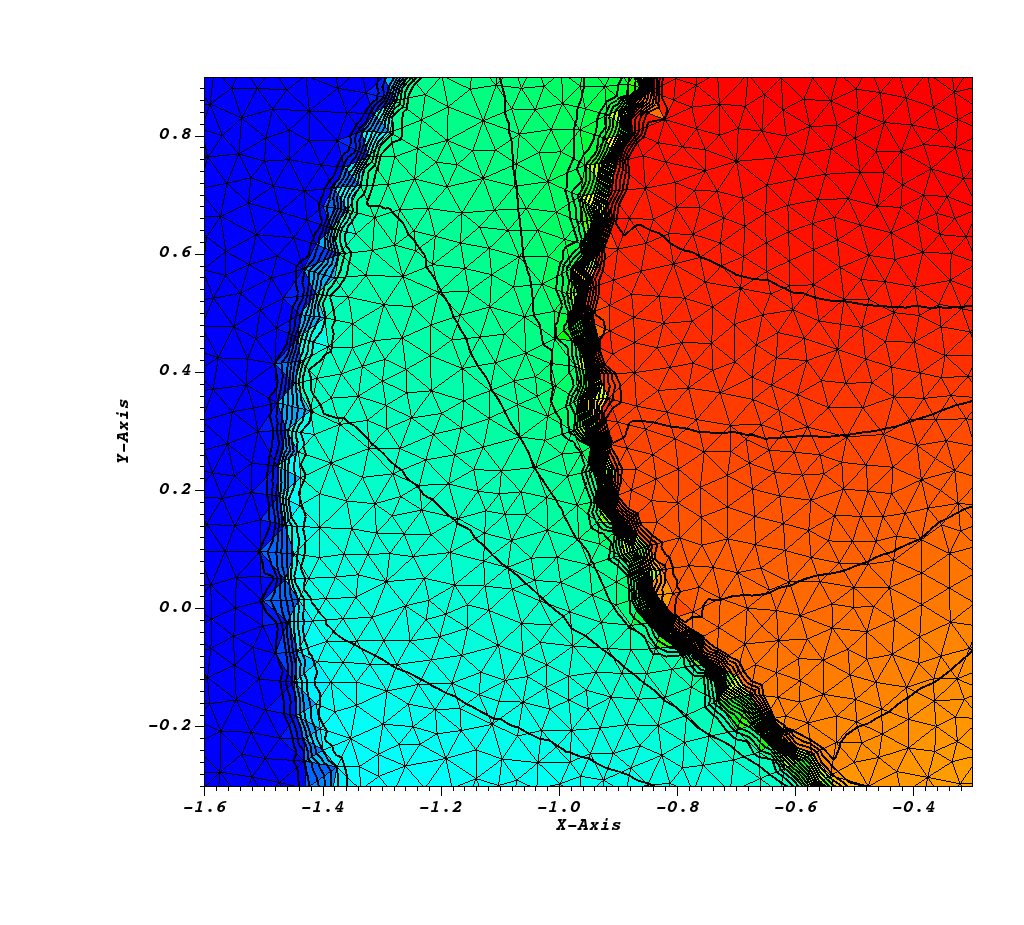}}
       \subfigure[point,Rusanov]{ \includegraphics[width=0.45\textwidth]{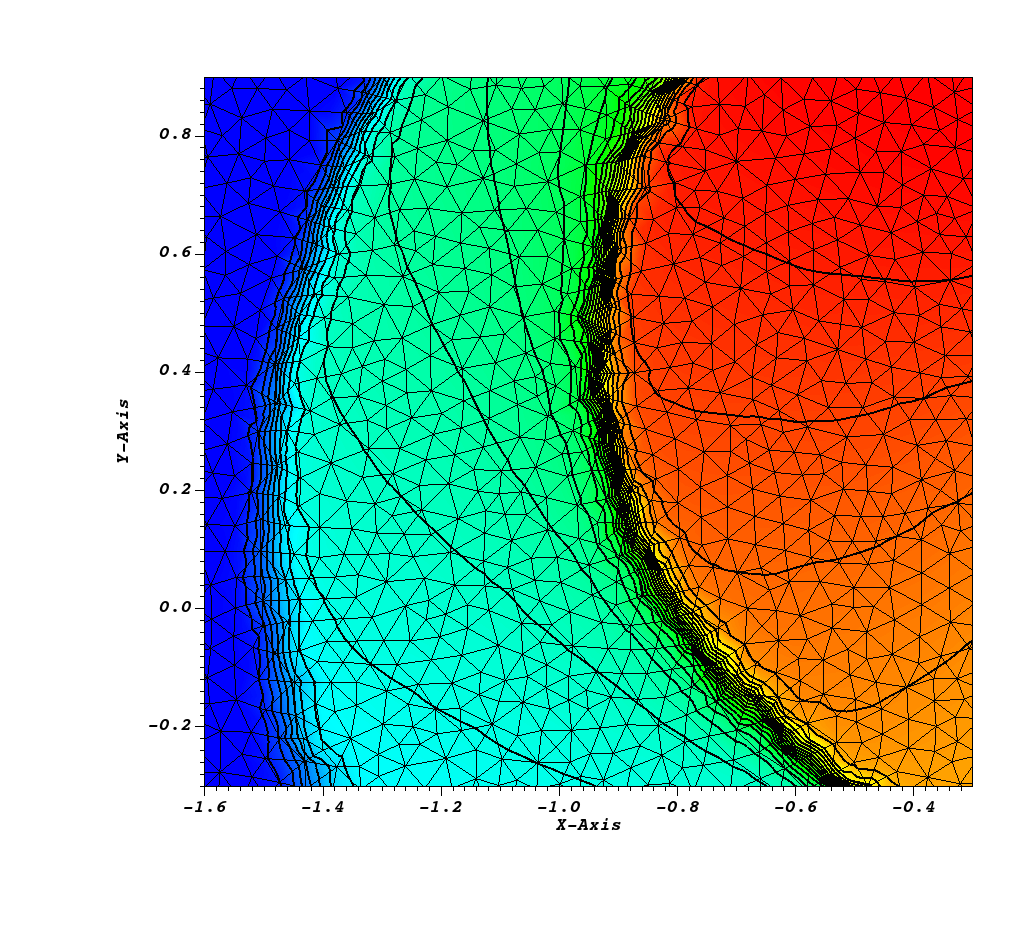}}
       \subfigure[ave,Rusanov]{ \includegraphics[width=0.45\textwidth]{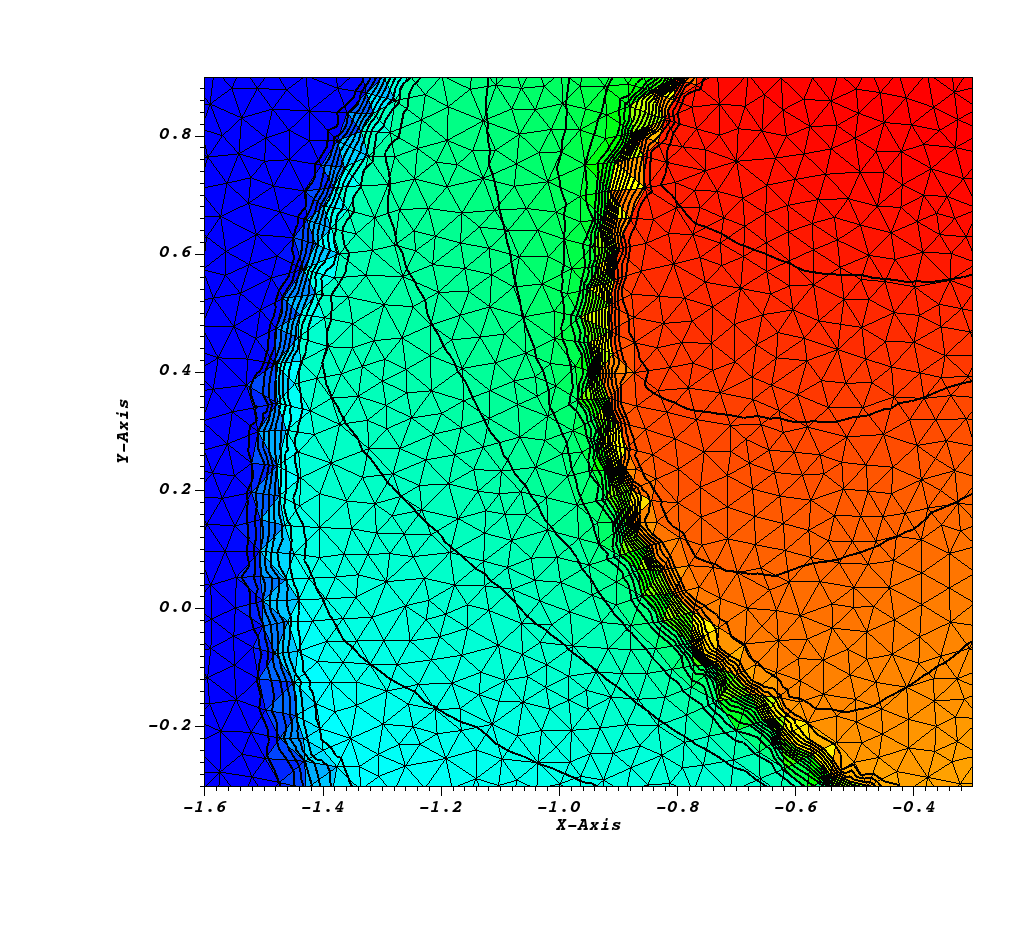}}
    \end{center}
    \caption{\label{fig:KPP:zoom}KPP problem, zoom of the mesh and the isovalues of the solution for: Tadmor correction only,  BP+OE+Tadmor and Rusanov soluton.}
    \end{figure}
\begin{table}[h]
\begin{center}
\begin{tabular}{| l || r|r||r|r|}
\hline
&\multicolumn{2}{c||}{Point values}&\multicolumn{2}{c|}{average}\\
\hline
Scheme & Min & Max & Min & Max\\
\hline
BP &$0.7854$&$11.$&$0.7854$&$11$\\
BP OE&$0.7854$&$11.$&$0.7854$&$11$\\
BP OE Tadmor&$0.7854$&$11.$&$0.7854$&$11$\\
Tadmor&$0.6719$ & $11.04$ & $0.6265$ &$11.13$\\
\hline
\end{tabular}
\end{center}
\caption{Min and max of the KPP solution for the different schemes.}
\end{table}

\section{Conclusions}
We have shown how to construct an entropy stable Active flux type scheme. Adding this condition to Bound preserving conditions and non oscillatory scheme, we get a scheme that is entropy stable, non oscillatory and bound preserving. Though our numerical examples are only given for triangular type meshes, the results will go without modification to polygonal type ones. We also have shown that nothing need to be done for the point value update. The structure of the proof shows that this result will also go to very high order active flux type schemes, as in \cite{BarsukowAbgrall,barsukow2025generalizedactivefluxmethod} with the same methodology: only act on average values.

We have (yet) not been able to show analytically that the entropy correction we propose is compatible with the preservation of the formal accuracy. However, the results, such as those of Figure \ref{fig:KPP:zoom}-(a), (b), (c) and (d), compared to \ref{fig:KPP:zoom}-(e) and (f), indicate that the scheme is much less dissipative than the baseline one, the Rusanov scheme.

\section*{Acknowledgements.} YL is supported by  SNSF grant 200020\_204917. Most of this research has been done while RA was visiting Professor E. Sonnendr\"ucker at the Max Plank Institute for Plasma Physics in G\"arching bei M\"unchen, Germany thanks to the support of the Alexander von Humboldt-foundation and the  Carl Friedrich von Siemens foundation, Germany.

\printbibliography
\appendix

\section{Proof of Lemma \protect{\ref{sec:rd:lemme1}}}\label{sec:appendix}

\begin{proof}[Proof of lemma \ref{sec:rd:lemme1}]
    We show that $$
    \lim_{h\rightarrow 0}\bigg (\sum_{n=0}^N\Delta t\sum_{K\subset Q}\sum_{\bsigma\in K} \Vert \bbu_\bsigma^{n+1}-\bbu_\bsigma^n\Vert \bigg )=0$$
    and $$\lim_{h\rightarrow 0}\bigg (\sum_{n=0}^N\Delta t\sum_{K\subset Q}\sum_{\bsigma\in K}\Vert \bbu_{\bsigma}^n-\xbar{\bbu}_K^n\Vert\bigg )=0$$
    First, there exists $C_1$ depending only on the family of meshes such that
    $$\sum_{n=0}^N \sum_{K\subset Q} \vert K\vert h \sum_{\bsigma\in K}\Vert \bbu_\bsigma^{n+1}-\bbu_\bsigma^n\Vert \leq \sum_{n}^N h\int_K\Vert \bbu_h(\bbx, t+\Delta t)-\bbu_h(\bbx,t)\Vert d\bbx\leq \int_0^T\int_Q\Vert \bbu_h(\bbx, t+\Delta t)-\bbu_h(\bbx,t)\Vert d\bbx dt$$
    because $\Delta t\leq C h$ since the family of solution is bounded and the mesh regular.
    Since the family $\{\bbu_h(\bullet, \bullet)\}$ is bounded, there exists $\bbv_1\in L^2(Q\times [0,T]$ such that $\bbu_k(\bullet, \bullet)\rightarrow \bbv_1$ in the $L^2$ weak-$\star$ topology. Similarly, there exists $\bbv_2\in L^2(Q\times [0,T])$ such that 
    $\bbu_k(\bullet, \bullet+\Delta t)\rightarrow v_2$ in the $L^2$ weak-$\star$ topology.
    Since $\bbu_h\rightarrow \bbv$ in $L^2(Q\times [0,T]$, $Q$ bounded and $L^2\subset L^1$ topologicaly, there exists a subsequence $\bbu_{h_n}$ such that $\bbu_{h_n}\rightarrow \bbv$ almost everywhere. Hence, using again the boundedness of $\bbu_h$ and the Lebesgue dominated convergence theorem, we see that $\bbv_1=\bbv_2=\bbv$, so that
    $$\int_0^T\int_Q\Vert \bbu_h(\bbx, t+\Delta t)-\bbu_h(\bbx,t)\Vert d\bbx dt\rightarrow0 .$$
    This gives the first part. The second part is obtained using the same arguments together with
    $$\sum_{\bsigma, \bsigma'\in K}\Vert \bbu_\bsigma^n-\bbu_{\bsigma'}^n\Vert +\sum_{\bsigma\in K}\Vert \bbu_\bsigma^n-\xbar{\bbu}_K^n\Vert \leq C\sum_{\bsigma\in K}\Vert \bbu_\bsigma^n-\xbar{\bbu}_K^n\Vert
    $$
    thanks to the triangle inequality applied to the first term and $C$ comes again from the regularity of the mesh.
\end{proof}
\end{document}